\documentclass[12pt,reqno]{amsart}

\usepackage{amscd,amsmath}
\usepackage{mathrsfs}
\usepackage{amsfonts}
\usepackage{amssymb}
\usepackage{enumerate}
\usepackage{setspace}
\usepackage{color}
\usepackage{nicefrac}
\usepackage[normalem]{ulem}

\usepackage[colorlinks=true,
            linkcolor=red,
            urlcolor=blue,
            citecolor=gray]{hyperref}

\usepackage{amsthm}
\usepackage{amssymb,latexsym,graphics,enumerate}
\usepackage[mathscr]{eucal}
\usepackage{amsmath,amsfonts,amsthm,amssymb}
\usepackage[left=1in,right=1in,top=1in,bottom=1in]{geometry}
\usepackage{tikz} 

\newcommand{\be}{\begin{equation}} 
\newcommand{\bel}{\begin{eqnarray}}
\newcommand{\ee}{\end{equation}} 
\newcommand{\eel}{\end{eqnarray}}
\newcommand{\ba}{\begin{aligned}}
\newcommand{\ea}{\end{aligned}}

\newcommand{\bs}{\backslash}

\numberwithin{equation}{section}
\newtheorem{theorem}{Theorem}[section]
\newtheorem{lemma}[theorem]{Lemma}

\newtheorem{proposition}[theorem]{Proposition}

\newtheorem{remark}[theorem]{Remark}
\newtheorem{assumption}{Assumption}[section]

\renewcommand{\geq}{\geqslant}
\renewcommand{\ge}{\geqslant}
\renewcommand{\leq}{\leqslant}
\renewcommand{\le}{\leqslant}

\def\bs{{\mathbf s}}

\def\diam{D}

\numberwithin{equation}{section}

\newcommand{\rd}{\,\mathrm{d}}

\def\bu{{\bf u}}
\def\bx{{\bf x}}
\def\by{{\bf y}}
\def\bz{{\bf z}}

\def\bv{{\bf v}}

\def\bs{{\bf s}}

\def\Et{E}
\def\Ei{P}
\def\cEt{\delta \Et}
\def\cEi{\delta \Ei}

\def\Ckitty{ C_{\infty}}
\def\Ckittyzero{ C_{s}}
\def\Cmeow{ C_{*}}
\def\Cmi{ C_{max}}
\def\CA{ C_A}
\def\Ccat{ C_{-}}
\def\Cplus{C_{+}}

\def\buin{{\mathbf u}_0} 
\def\uin{u_0} 
\def\rin{\rho_0} 
\def\Etin{{\Et_0}}
\def\Eiin{{\Ei_0}}
\def\cEtin{{\cEt_0}}
\def\cEiin{{\cEi_0}}
\def\FP{U} 
\def\diam{D} 
\def\mixed{\lambda} 

\begin{document}

\title[Flocking hydrodynamics with external potentials]
{Flocking hydrodynamics with external potentials}

\author{Ruiwen Shu}
\address{Department of Mathematics and Center for Scientific Computation and Mathematical Modeling (CSCAMM)\newline
University of Maryland, College Park MD 20742}
\email{rshu@cscamm.umd.edu}

\author{Eitan Tadmor}
\address{Department of Mathematics, Center for Scientific Computation and Mathematical Modeling (CSCAMM), and Institute for Physical Sciences \& Technology (IPST)\newline
University of Maryland, College Park MD 20742}
\email{tadmor@cscamm.umd.edu}

\date{\today}

\subjclass{92D25, 35Q35, 76N10}

\keywords{flocking, hydrodynamics, hypocoercivity, harmonic oscillator, regularity, critical thresholds.}

\thanks{\textbf{Acknowledgment.} Research was supported in part by NSF grants DMS16-13911, RNMS11-07444 (KI-Net) and ONR grant N00014-1812465.}
\date{\today}

\begin{abstract}
We study the large-time behavior of  hydrodynamic model  which describes the collective behavior of continuum of agents, driven by  pairwise alignment interactions  with additional external potential forcing. The external force tends to compete with alignment  which makes the large time behavior very different from the original Cucker-Smale (CS) alignment model, and far more interesting.  Here we focus on  uniformly convex potentials. In the particular case of \emph{quadratic} potentials, we are able to treat a large class of admissible interaction kernels, $\phi(r) \gtrsim (1+r^2)^{-\beta}$ with  `thin' tails $\beta \leq 1$ --- thinner than the usual `fat-tail' kernels encountered in CS flocking $\beta\leq\nicefrac{1}{2}$: we discover unconditional flocking with exponential  convergence of velocities \emph{and} positions  towards a  Dirac mass traveling as harmonic oscillator. For general convex potentials, we impose a  stability condition, requiring large enough   alignment kernel to avoid crowd  scattering.   We then prove, by hypocoercivity arguments, that  both the velocities \emph{and} positions of smooth solution must flock.   We also prove  the existence of global smooth solutions for one and two space dimensions, subject to critical thresholds in initial configuration space. It is interesting to observe that global smoothness can  be guaranteed for sub-critical initial data, independently of the apriori knowledge of large time flocking behavior.
\end{abstract}

\maketitle
\setcounter{tocdepth}{1}
\tableofcontents

\section{Introduction}

We are concerned with the hydrodynamic alignment model with external potential forcing:
\begin{equation}\label{eq}
\left\{\begin{split}
& \partial_t \rho + \nabla_\bx \cdot(\rho \bu) = 0, \\
& \partial_t \bu + \bu\cdot\nabla_\bx \bu = \int \phi(|\bx-\by|)(\bu(\by,t)-\bu(\bx,t))\rho(\by,t)\rd{\by} - \nabla\FP(\bx).
\end{split}\right.
\end{equation}
Here $(\rho(\bx,t),\bu(\bx,t))$ are the local density and velocity field of a continuum of agents, depending on the spatial variables $\bx\in\Omega = \mathbb{R}^d \text{ or } \mathbb{T}^d$  and time $t\in\mathbb{R}_{\ge 0}$. 
The integral term on the right represents the alignment between agents, quantified in terms of the pairwise interaction kernel  $\phi=\phi(r)\ge 0$.
In many realistic scenarios,   agents driven by  alignment  are also subject to other forces  --- external forces from environment, pairwise attractive-repulsive forces, etc. Such forces may {\it compete} with alignment, which makes the large time behavior very different from the original potential-free model and far more interesting. One of the simplest type of external forces is {\it potential force}, given by the fixed external potential $\FP(\bx)$ on the right of \eqref{eq}. This is the main topic on the current work.

The system \eqref{eq} is a realization of the large-crowd dynamics
of the agent-based  system in which $N\gg1 $ agents identified with their position and velocity pair, $(\bx_i(t),\bv_i(t)) \in (\Omega \times \mathbb{R}^d)$, are driven by Cucker-Smale (CS) alignment \cite{CS2007a,CS2007b}, with additional external potential force
\begin{equation}\label{CS_par1}
\left\{\begin{split}
& \dot{\bx}_i = \bv_i \\
& \dot{\bv}_i = \frac{1}{N}\sum_{j\ne i} \phi(|\bx_i - \bx_j|)(\bv_j-\bv_i) - \nabla\FP(\bx_i)
\end{split}\right.\quad i=1,\dots,N.
\end{equation}
In the absence of any other forcing terms, both the agent-based system \eqref{CS_par1} and its large crowd description  \eqref{eq} have been studied intensively in the recent decade.
The most important feature of the potential-free CS model, \eqref{CS_par1} with $\FP\equiv 0$, is its {\it flocking} behavior: for a large class of interaction kernels satisfying the  `fat tail' condition, 
\begin{equation}\label{eq:fat}
 \int_0^\infty \phi(r)\rd{r}=\infty,
 \end{equation}
  \emph{global} alignment of velocities follows \cite{HT2008,HL2009}, 
$|\bv_i(t)-\bv_j(t)|\stackrel{t\rightarrow \infty}{\longrightarrow} 0$.
The presence of additional potential forcing in  the one-dimensional discrete system \eqref{CS_par1} was  recently studied in~\cite{HS2018}, where it is shown  that at least for some special choices of $\FP$,  {\it both position and velocity} align for large time, $|\bv_i(t)-\bv_j(t)|+ |\bx_i(t)-\bx_j(t)|\stackrel{t\rightarrow \infty}{\longrightarrow} 0$.\newline
The corresponding potential-free continuum system, \eqref{eq} with $\FP\equiv 0$, was studied in  \cite{HT2008, HL2009, CFTV2010, MT2014}: the large time behavior of its smooth solutions is captured by flocking, $|\bu(\bx,t)-\bu(\by,t)|\rho(\bx)\rho(\by)\stackrel{t\rightarrow \infty}{\longrightarrow} 0$, similar to the underlying discrete system. Moreover,
existence of one- and two-dimensional global smooth solutions was proved for a large class of initial configurations which satisfy certain critical threshold condition, \cite{TT2014,CCTT2016,ST2017a,ST2017b,HeT2017} and general multiD problems with nearly aligned initial data \cite{Sh2018, DMPW2018}.

\smallskip\noindent
 In this paper we study  the alignment dynamics  in the $d$-dimensional continuum system \eqref{eq}. 
We focus on the following two key aspects of \eqref{eq}.

$\bullet$ {\bf The flocking phenomena of global smooth solutions}, if they exist. Such results are well known in the absence of external potential --- smooth solutions subject to pure alignment must flock \cite{HT2008,TT2014,HeT2017}, but the presence of external potential has a confining effect which competes with alignment. Here we explore the flocking phenomena in the presence of
\emph{uniformly convex} potentials 
\be\label{eq:uconvex}
aI_{d\times d}\leq \nabla^2\FP(\bx) \leq AI_{d\times d}, \qquad 0<a<A.
\ee
The upper-bound on the right is \emph{necessary} for existence of  1D  global smooth solutions, consult theorems \ref{thm_1dsmooth}--\ref{thm_1dblowup} below; the uniform convexity on the left is necessary for the flocking behavior. We discover, in section \ref{sec:convex}, that  both the velocities \emph{and} positions of smooth solution must flock  at algebraic rate under a linear stability condition (\ref{Kcond1}),  
$m_0\phi(0) > \frac{A}{\sqrt{a}}$.
The necessity of a precise   stability condition, at least in  the general convex case, remains open. 
We can be much more precise in the special case of  \emph{quadratic potentials}, 
\be\label{eq:uquad}
\FP(\bx)=\frac{a}{2}|\bx|^2, \qquad a>0.
\ee
Here, in section \ref{sec:quadratic}, we discover unconditional flocking of velocities and positions with   \emph{exponential} convergence to a  Dirac mass traveling as a harmonic oscillator.
Moreover, the confining effect of the quadratic potential applies to  interaction kernels, $\phi(r) \gtrsim (1+r^2)^{-\beta}$ which allow for `thin' tails $\beta \leq 1$ --- thinner than the usual `fat-tail' kernels encountered in CS flocking \eqref{eq:fat}. 

\smallskip
$\bullet$ {\bf Existence of global smooth solutions}. In the absence of external force,  the existence of global smooth solutions of the  one- and respectively two-dimensional \eqref{eq} was proved in \cite{TT2014,CCTT2016}   and respectively \cite{HeT2017}, provided the initial data is `below'  certain critical threshold expressed in terms of the initial data $\nabla\buin$. We mention in passing that in case of singular kernel $\phi$, then smooth solutions exist independent of an initial threshold \cite{ST2017a}).  In the presence of additional convex  potential, \eqref{eq:uconvex}, we discover that the critical thresholds still exist, though they are tamed by the presence of $\FP$ (consult \cite{TW2008}).  In the particular case of quadratic potential \eqref{eq:uquad},  $\FP(\bx)$ does not  affect the dynamics of the spectral gap of $\nabla_S\bu$ which is a crucial step of the  regularity result in~\cite{HeT2017}, leading to existence of global smooth solutions. 
Existence with general  convex  potentials \eqref{eq:uconvex} requires different methodology than the quadratic case. These results are summarized in section \ref{sec:existence}.

\section{Statement of main results --- flocking with quadratic potentials}\label{sec:quadratic}
We focus attention to 
\emph{quadratic potentials}, $\displaystyle \FP(\bx)=\frac{a}{2}|\bx|^2$,
where \eqref{eq} reads 
\begin{equation}\label{eq:quadratic}
\left\{\begin{split}
& \partial_t \rho + \nabla_\bx \cdot(\rho \bu) = 0, \\
& \partial_t \bu + \bu\cdot\nabla_\bx \bu = \int \phi(|\bx-\by|)(\bu(\by,t)-\bu(\bx,t))\rho(\by,t)\rd{\by} - a\bx.
\end{split}\right.
\end{equation}

\subsection{General considerations} We begin by recording general observations on system \eqref{eq} which is subject to  sufficiently smooth  data $(\rin,\buin)$, such that $\rin\geq0$ is compactly supported. Denote the total mass
\[
m_0 := \int \rin(\bx)\rd{\bx} >0.
\]

\medskip
\paragraph{{\bf $\bullet$ Interaction kernels}}
We assume that the system \eqref{eq} is driven by an  interaction kernel from a general class of \emph{admissible kernels}.
\begin{assumption}[{\bf Admissible kernels}]\label{phicond}
We consider \eqref{eq} with interaction kernel $\phi$ such that
\begin{subequations}\label{eqs:admiss}
\begin{align}
&\hspace*{-1.0cm}\mbox{(i)}   \quad \phi(r) \mbox{ is positive, decreasing and  bounded}: \ 
0< \phi(r) \le \phi(0):=\phi_+ < \infty; \label{phicond0}\\
&\hspace*{-1.0cm}\mbox{(ii)} \quad \phi(r) \mbox{ decays slow enough at infinity in the sense that } \int^\infty \!\!r\phi(r) \rd{r} =\infty. \label{phicond1}
\end{align}
\end{subequations}
\end{assumption}
\noindent
Note that (\ref{phicond1}) allows a larger admissible class of $\phi$'s with {\it thinner} tails than the usual `fat-tail' assumption \eqref{eq:fat} which characterizes unconditional  flocking of potential-free alignment, e.g., the original choice of Cucker-Smale, $\phi(r)=(1+r^2)^{-\beta},\, \beta \leq\nicefrac{1}{2}$ is now admissible for the improved range $\beta \leq 1$. 

\medskip
\paragraph{{\bf $\bullet$ Harmonic oscillators}} 
The distinctive feature of the alignment dynamics with quadratic potential \eqref{eq:quadratic}, is its Galilean invariance  w.r.t. the dynamics of harmonic oscillator associated with \eqref{eq:quadratic}. Thus,  let $(\bx_c,\bu_c)$ denote the mean position and the mean velocity 
\begin{subequations}\label{eqs:xcuc}
\begin{equation}\label{eq:xcuc}
\left\{\begin{split}
& \bx_c(t):=\frac{1}{m_0}\int \bx\rho(\bx,t)\rd{\bx} \\
& \bu_c(t):=\frac{1}{m_0}\int \bu(\bx,t)\rho(\bx,t)\rd{\bx}; \\
\end{split}\right.
\end{equation}
by \eqref{eq:quadratic}, these means are governed by the harmonic oscillator
\begin{equation}\label{xcuc}
\left\{\begin{split}
& \dot{\bx}_c = \bu_c \\
& \dot{\bu}_c = -a\bx_c. 
\end{split}\right.
\end{equation}
\end{subequations}
The translated quantities centered around  
the means, $\widehat{\rho}(\bx,t) = \rho(\bx_c(t)+\bx,t)$ and $\widehat{\bu}(\bx,t) = \bu(\bx_c(t)+\bx,t)-\bu_c(t)$, 
satisfy the same system \eqref{eq:quadratic} with vanishing mean location and mean velocity. We can therefore assume without loss of generality, after re-labeling $(\widehat{\rho},\widehat{\bu})\leadsto (\rho, \bu)$, that the solution of \eqref{eq:quadratic}  satisfies
\begin{equation}\label{eq:mean}
\int \bx{\rho}(\bx,t)\rd{\bx} \equiv 0, \quad \int{\bu}(\bx,t){\rho}(\bx,t)\rd{\bx}\equiv 0, \quad \mbox{for all} \ \   t\geq0.
\end{equation}

\medskip
\paragraph{{\bf $\bullet$ Energy decay}} We record below the basic energy bounds with general external potentials. Let  $\Et(t)$ denote the \emph{total energy} associated with \eqref{eq},
\begin{equation}\label{energy}
\Et(t) := \int \left(\frac{1}{2}|\bu(\bx,t)|^2+\FP(\bx)\right) \rho(\bx,t)\rd{\bx}
\end{equation}
The fundamental bookkeeping of \eqref{eq} is  the $L^2$-energy decay 
\be\label{eq:decay}
\frac{\rd}{\rd t} \Et(t) = -\frac{1}{2}\int\int \phi(|\bx-\by|)|\bu(\bx,t)-\bu(\by,t)|^2\rho(\bx,t)\rho(\by,t)\rd{\bx}\rd{\by}
\ee
This relates the  decay \emph{rate} of the energy  to the enstrophy, quantified in terms of \emph{energy fluctuations} on the right. We emphasize that the bound \eqref{eq:decay} applies to general  external potentials $\FP$.

\subsection{Bounded support}
A priori estimates for the growth rate of the support of $\rho$ is the key for proving flocking results for admissible kernels  $\phi$ with proper decay at infinity. For the case without external potential, it is straightforward to show that the velocity variation $\max_{t\ge 0,\,\bx,\by\in\textnormal{supp\,}\rho(\cdot,t)}|\bu(\bx,t)-\bu(\by,t)|$ is non-increasing, which implies the linear growth, $\textnormal{diam}(\textnormal{supp\,}\rho(\cdot,t)) = {\mathcal O}(t)$ which in turn yields the `fat-tail' condition \eqref{eq:fat}. Here we show that confining effect of the external potential enforces  the support of $\rho(\cdot,t)$ to remain \emph{uniformly bounded}. 

To this end, define the maximal particle energy
\begin{equation}
\Ei(t) := \max_{\bx\in\textnormal{supp\,}\rho(\cdot,t)} \Big(\frac{1}{2}|\bu(\bx,t)|^2 + \FP(\bx)\Big).
\end{equation}
The confinement effect of the external potential shows that this $L^\infty$-particle energy remains uniformly bounded in time.
We then `pair' the quadratic growth of $\FP(\bx)$ with the admissibility of thin-tails assumed in \eqref{phicond1}, to show that $\textnormal{supp\,}\rho(\cdot,t)$ remains uniformly bounded.

\begin{lemma}[{\bf Uniform bounds on particle energy}]\label{lem_grow}
Let $(\rho,\bu)$ be a smooth solution to \eqref{eq:quadratic} with an admissible interaction kernel \eqref{eqs:admiss}. Then  the particle energy and hence the support of $\rho(\cdot,t)$ remain uniformly bounded
\begin{equation}\label{grow}
\frac{a}{8}\diam^2(t) \leq \Ei(t) \le R_0, \qquad \diam(t):=\textnormal{diam}(\textnormal{supp\,}\rho(\cdot,t)).
\end{equation}
Here, the spatial scale $R_0=R_0(\phi_+,m_0,a,\Etin,\Eiin)$ is dictated by \eqref{eq:R0} below.
\end{lemma}
For the proof, follow the  particle energy $F(\bx,t) := \frac{1}{2}|\bu(\bx,t)|^2+ \FP(\bx)$ along characteristics,
\[
\begin{split}
F' = & \partial_t F + \bu\cdot\nabla F \\
= & \bu\cdot\left(-\bu\cdot\nabla\bu + \int \phi(\bx-\by)(\bu(\by)-\bu(\bx))\rho(\by)\rd{\by} - \nabla \FP(\bx)\right)  + \bu\cdot(\bu\cdot\nabla\bu)  + \bu\cdot\nabla\FP(\bx) \\
= & \bu\cdot\left(  \int \phi(\bx-\by)(\bu(\by)-\bu(\bx))\rho(\by)\rd{\by}\right)  \\
= &  \int \phi(\bx-\by)(\bu(\bx)\cdot\bu(\by)-|\bu(\bx)|^2)\rho(\by)\rd{\by} \\
= &  \int \phi(\bx-\by)\Big(-\frac{1}{4}|\bu(\by)|^2 + \bu(\bx)\cdot\bu(\by)-|\bu(\bx)|^2\Big)\rho(\by)\rd{\by} + \int \phi(\bx-\by)\frac{1}{4}|\bu(\by)|^2 \rho(\by)\rd{\by} \\
= &  -\int \phi(\bx-\by)|\bu(\bx)-\frac{1}{2}\bu(\by)|^2\rho(\by)\rd{\by} + \frac{1}{4}\int \phi(\bx-\by)|\bu(\by)|^2 \rho(\by)\rd{\by} 
\le   \frac{\phi_+}{2}E_k(t),
\end{split}
\]
where $E_k(t)$ denotes the \emph{kinetic energy}
\be\label{eq:Ek}
\frac{\rd}{\rd t} \Ei(t) \leq \frac{\phi_+}{2}E_k(t), \qquad E_k(t):=\frac{1}{2}\int |\bu(\bx,t)|^2\rho(\bx,t)\rd{\bx}.
\ee
We emphasize that the bound \eqref{eq:Ek} applies to general symmetric kernels $\phi$ and is otherwise independent of the fine structure of the potential $\FP$.
Recalling  the diameter $\diam(t)=\textnormal{diam}(\textnormal{supp\,}\rho(\cdot,t))$, then 
 $L^2$-energy decay \eqref{eq:decay} yields 
\[
\frac{\rd}{\rd t}\Et(t) \leq -\frac{1}{2}\phi(\diam(t))\int\!\!\int |\bu(\bx,t)\!-\!\bu(\by,t)|^2\rho(\bx,t)\rho(\by,t)\rd{\bx}\rd{\by},
\]
and in view of \eqref{eq:mean}, this decay rate can be formulated in terms of the kinetic energy
\be\label{eq:Edecay}
\frac{\rd}{\rd t}\Et(t) \leq -2m_0\phi(\diam(t))E_k(t).
\ee
Further, the support of $\rho(\cdot,t)$ can be bounded in terms of the particle energy we have
\be\label{eq:PandD}
\Ei(t) \geq U(x) = \frac{a}{2}\max_{\textnormal{supp\,}\rho(\cdot,t)}|\bx|^2 \geq \frac{a}{8}\diam^2(t), \qquad \diam(t)=\textnormal{diam}(\textnormal{supp\,}\rho(\cdot,t)).
\ee
Finally, by the fat-tail assumption \eqref{phicond1}, 
$\displaystyle 
\int^\infty \phi(\sqrt{8r/a})\rd{r} = \frac{a}{4}\int^\infty r\phi(r)\rd{r}=\infty$, there exists a finite spatial scale $R_0>\Eiin$ such that 
\be\label{eq:R0}
\int^{R_0}_{\Eiin}\phi(\sqrt{8r/a})\rd{r} > \frac{\phi_+}{4m_0}\Etin.
\ee
We now consider the functional  $\displaystyle Q(t):=\Et(t)+ \frac{4m_0}{\phi_+}\int^{\Ei(t)}_{R_0}\!\!\!\!\!\!\phi(\sqrt{8r/a})\rd{r}$ which we claim is non-positive: indeed,  by \eqref{eq:R0}, $Q(0)\leq 0$ and in view of \eqref{eq:Ek}--\eqref{eq:PandD},  $Q(t)$  decreasing in time
\[
\frac{\rd}{\rd t}Q(t) \leq -2m_0\phi(\diam(t))E_k(t) + \frac{4m_0}{\phi_+}\frac{\phi_+}{2}E_k(t)\times \phi(\sqrt{8\Ei(t)/a}) \leq0.
\]
It follows that the particle energy remains uniformly bounded, 
\[
\frac{4m_0}{\phi_+}\int^{\Ei(t)}_{R_0}\!\!\!\phi(\sqrt{8r/a})\rd{r} \leq Q(t) \leq
0,
\]
hence $\Ei(t)$ remain bounded, $\Ei(t)\leq R_0$, and the uniform bound on $\diam(t)$ stated in  \eqref{grow} follows from \eqref{eq:PandD}. \qed\mbox{ }

\medskip\noindent
For the typical example of  $\phi(r)= c_0(1+r^2)^{-\beta}$ we find that \eqref{eq:R0} holds with
\[
R_0  \geq \frac{a}{8}\left[ \Big(\big(1+\frac{8}{a}P_0\big)^{1-\beta} + \frac{2(1-\beta)\phi_+}{a c_0 m_0}E_0 \Big)^{\frac{1}{1-\beta}} - 1\right].
\]

\medskip
\begin{remark}[{\bf On quadratic potential and pairwise interactions}]\label{rem:pair}
We emphasize that the proof of lemma \ref{lem_grow} relies on the special structure of the quadratic potential, namely, the Galilean invariance  with respect to harmonic oscillator \eqref{xcuc} which no longer holds for a general potentials. Specifically, observe that by the Galilean invariance, the energy decay rate  \eqref{eq:decay}  in terms of energy \emph{fluctuations} is converted into the $L^2$-energy decay \eqref{eq:Edecay}.\newline
We close this section by noting that the same Galilean invariance is intimately related to the fact that quadratic external forcing can be interpreted as \emph{pairwise interactions},
\begin{equation}\label{CS_pair}
\left\{\begin{split}
& \dot{\bx}_i = \bv_i \\
& \dot{\bv}_i = \frac{1}{N}\sum_{j\ne i} \phi(|\bx_i - \bx_j|)(\bv_j-\bv_i) - \frac{a}{N} \sum_{j\ne i} (\bx_i-\bx_j).
\end{split}\right.
\end{equation}
Indeed,  since the averages for the solution to \eqref{CS_par1} with $\FP=\frac{a}{2}|\bx|^2$--- the center of mass $\bx_c(t):=\nicefrac{1}{N}\sum_i \bx_i$ and mean velocity $\bu_c(t):=\nicefrac{1}{N}\sum_i \bv_i$ satisfy \eqref{xcuc}, we find that the translated quantities
$\bx_i \mapsto \bx_i-\bx_c(t),\,\bv_i \mapsto \bv_i-\bu_c(t)$ satisfy  \eqref{CS_pair}.  Similarly, the large crowd dynamics associated with \eqref{CS_pair} 
\begin{equation}\label{eq:pair}
\hspace*{-0.7cm}\left\{\begin{split}
& \hspace*{0.1cm} \partial_t \rho + \nabla_\bx \cdot(\rho \bu) = 0, \\
& \hspace*{0.1cm} \partial_t \bu + \bu\cdot\nabla_\bx \bu\!=\!\int \!\phi(|\bx-\by|)(\bu(\by,t)-\bu(\bx,t))\rho(\by,t)\rd{\by}\!-\!\frac{a}{m_0}\int \!(\bx-\by)\rho(\by,t)\rd{\by},
\end{split}\right.
\end{equation}
coincides with \eqref{eq:quadratic}  under suitable Galilean variable transformation.
\end{remark}

\ifx
\subsection{A priori estimate for the support of $\rho$}

A priori estimates for the growth rate of the support of $\rho$ is the key to prove flocking results with $\phi$ decaying at infinity. For the case without external potential, it is straightforward to show that the velocity variation $\max_{t\ge 0,\,\bx,\by\in\textnormal{supp\,}\rho(\cdot,t)}|\bu(\bx,t)-\bu(\by,t)|$ is non-increasing, which implies the linear growth of the support of $\rho$, i.e., $\textnormal{diam }(\textnormal{supp\,}\rho(\cdot,t)) = O(t)$. In the presence of confining external potential, we improve this estimate in three major stages. 

The first stage is a direct consequence of the particle energy bound \eqref{grow}, implying that $\textnormal{diam}(\textnormal{supp\,}\rho(\cdot,t)) \lesssim \sqrt{1+t}$. 
\begin{theorem}\label{thm_grow}
Let $(\rho,\bu)$ be a smooth solution of \eqref{eq} with potential  $\FP(\bx)\ge \frac{a}{2}|\bx|^2, \ a >0$. Then
\begin{equation}\label{grow2}
\textnormal{diam}(\textnormal{supp\,}\rho(\cdot,t)) \le C_a\sqrt{\Eiin + \frac{1}{2}\phi_+\Etin t}, \qquad C_a= \frac{2\sqrt{2}}{\sqrt{a}}.
\end{equation}
\end{theorem}
Indeed, for $\FP(\bx)\ge \frac{a}{2}|\bx|^2$, (\ref{grow}) implies
\[
\max_{\bx\in\textnormal{supp\,}\rho(\cdot,t)} |\bx| \le \frac{\sqrt{2}}{\sqrt{a}}\sqrt{\max_{\bx\in\textnormal{supp\,}\rho(\cdot,t)} \FP(\bx)} \le\frac{C_a}{2}\sqrt{\Eiin + \frac{1}{2}\phi_+\Etin t},
\]
and (\ref{grow2}) follows.
\begin{remark}
Theorem \ref{thm_grow} reflects the confining aspect of super-quadratic potentials. It applies to  general interaction kernels $\phi(\bx,\by)$, as long as $\phi$ is non-negative, bounded above, and symmetric: $\phi(\bx,\by) = \phi(\by,\bx)$. For example, one can take $\phi$ to be compactly supported, or the topological interaction kernel defined in \cite{ST2018}. Notice that although the proof of theorem \ref{thm_grow} does not use the symmetry of $\phi$ directly, it does use the energy decay estimate in the estimate of $F'$, which requires the symmetry of $\phi$.
\end{remark}

In the second stage we `pair' the root-growth of the $\textnormal{supp\,}\rho(\cdot,t)$ in \eqref{grow2} with the admissibility of sub-quadratic thin-tales assumed in \eqref{phicond1}, to show that $\textnormal{supp\,}\rho(\cdot,t)$ remains uniformly bounded.
\begin{proposition}[{\bf Flocking with sub-exponential rate}]\label{prop_flocking01}
Let $(\rho,\bu)$ be a global smooth solution of  \eqref{eq:quadratic}, subject to compactly supported $\rin$. Then the following  flocking estimates at sub-exponential rate hold.
\begin{align}
\cEt(t)& := \int\!\!\int (|\bu(\bx)-\bu(\by)|^2 + a|\bx-\by|^2)\rho(\bx)\rho(\by) \rd{\bx}\rd{\by} \le C\cdot\cEt(0)e^{-\lambda_0 (1+t)^{1-\beta}},\label{flocking0}\\
\cEi(t) & := \max_{\bx,\by\in\textnormal{supp\,}\rho(\cdot,t)} (|\bu(\bx,t)-\bu(\by,t)|+\sqrt{a}|\bx-\by|) 
 \le \Ckittyzero \cdot \cEi(0) e^{-\lambda_0 (1+t)^{1-\beta}}. \label{cEinfty0}
\end{align}
Here $\Ckittyzero$ and $\lambda_0$ are positive constants depending on the data $a ,\phi_+,m_0, c_0, \beta, \Etin, \Eiin$ and and $C>0$ is an absolute constant. In particular,  there holds the uniform-in-time bound
\begin{equation}\label{grow3}
\textnormal{diam}(\textnormal{supp\,}\rho(\cdot,t)) \le \frac{\Ckittyzero}{\sqrt{a}}\cEi(0).
\end{equation}
\end{proposition}
The proof of the $L^2$-portion of this proposition is a standard application of hypo-coercivity, i.e., adding a small cross term to the energy estimate $\cEt$ to compensate the lack of dissipation in $\bx$. We then improve the $L^2$-flocking estimate which yields the  $L^\infty$-portion of the proposition. 

\medskip
\fi

\subsection{Flocking of smooth solutions with exponential rate}
The {\it uniform-in-time} bound on the $\textnormal{supp\,}\rho(\cdot,t)$in \eqref{grow} shows that the values $\phi(r)$ with $r > \sqrt{8R_0/a}$ play no role in the solution of (\ref{eq:quadratic}). We can therefore assume without loss of generality that our admissible $\phi$'s are uniformly bounded from below, 
\begin{equation}\label{eq:phimin}
\phi(r) \ge \phi(\diam(t)) \geq \phi_- > 0, \qquad \phi_-:= \phi\Big(\frac{\sqrt{8R_0}}{\sqrt{a}}\Big).
\end{equation}
This enables us  prove   our  main statement of flocking with exponential decay. 
\begin{theorem}[{\bf Flocking with  $L^2$-exponential decay}]\label{thm_flocking1}
Let $(\rho,\bu)$ be a global smooth solution of  \eqref{eq:quadratic}, subject to compactly supported $\rin$. Then there holds the flocking estimate at exponential rate in both velocity and position:
\begin{equation}\label{flocking}
\cEt(t) :=\int\!\!\int (|\bu(\bx,t)-\bu(\by,t)|^2 + a|\bx-\by|^2)\rho(\bx,t)\rho(\by,t) \rd{\bx}\rd{\by} \le 2\cdot\cEtin \cdot e^{-\mixed t}.
\end{equation}
Here $\mixed=\mixed(a ,\phi_-,\phi_+,m_0)>0$.
\end{theorem}

\begin{remark}
In fact, one could develop a small-data result, where the exponential flocking asserted in  theorem \ref{thm_flocking1} is extended to  $\FP$'s close to  quadratic potential provided  under appropriate smallness condition on the initial data.
\end{remark}

From the proof of theorem \ref{thm_flocking1}, one can take the decay rate
\begin{equation}\label{lambda}
\mixed = \mixed(a):=\frac{1}{2}\min\left\{ \frac{m_0\phi_-}{ \nicefrac{\displaystyle m_0^2\phi_+^2}{\displaystyle a} + \nicefrac{\displaystyle 3}{\displaystyle 2}}, \frac{\sqrt{a}}{2} \right\}
\end{equation}
If one fixes $m_0$, $\phi_+$, $\phi_-$ and considers the asymptotic behavior for $a\rightarrow 0$, then the decay rate $\mixed = {\mathcal O}(a)$. For $a\rightarrow \infty$, the decay rate $\mixed={\mathcal O}(1)$. This shows that {\it the strength of external potential force} may have significant influence on the rate of flocking, and a weak potential tends to give a slower decay. One could interpret this as follows: to achieve an equilibrium, both velocity and position have to align; if the potential force is weak, then the alignment of position happens on a slower time scale, since the potential-free Cucker-Smale interaction does not provide position alignment.

\smallskip\noindent
Next, we turn to improve the $L^2$-flocking estimate in theorem \ref{thm_flocking1} into an $L^\infty$ estimate:
\begin{theorem}[{\bf Flocking with uniform exponential decay}] \label{prop_infty1}
Let $(\rho,\bu)$ be a global smooth solution of  \eqref{eq:quadratic}, subject to compactly supported $\rin$.  Then 
\begin{equation}\label{cEinfty}
\cEi(t):= \max_{\bx,\by\in\textnormal{supp\,}\rho(\cdot,t)} (|\bu(\bx,t)-\bu(\by,t)|^2+a|\bx-\by|)^2   \le \Ckitty \cdot \cEiin \cdot e^{-\lambda t/2},\quad \forall t\ge 0
\end{equation}
where the decay rate  $\mixed=\mixed(a)>0$  given by \eqref{lambda} and 
$\Ckitty$ is a positive constant given by  
\[
\Ckitty = 4\Big(1+\phi_+^2 m_0^2\Big(\frac{2}{m_0\phi_- \mixed(a)}+\frac{4}{a}\Big)\Big).
\]
\end{theorem}
We conclude that the smooth solution of \eqref{eq:quadratic} converges exponentially to the harmonic oscillator \eqref{eqs:xcuc}
\begin{align}
\begin{split}
\rho(\bx,t)-m_0\delta(\bx-\bx_c(t)) &\stackrel{t \rightarrow \infty}{\longrightarrow} 0, \\
 \rho\bu(\bx,t)-m_0\bu_c(t)\delta(\bx-\bx_c(t)) & \stackrel{t \rightarrow \infty}{\longrightarrow} 0.
\end{split} 
\end{align}
Note that since 
$\displaystyle \cEt \le m_0^2\cdot \cEi$,
the $L^\infty$-version of flocking stated in  theorem \ref{prop_infty1} is an improvement of theorem \ref{thm_flocking1}: this improvement will be {\it crucial} in  studying the existence of global smooth solution for two-dimensional systems asserted in theorem \ref{thm_2dsmooth1} below. 

\ifx
\left[\begin{array}{c}\bx_c(t) \\ \frac{1}{\sqrt{a}}\bu_c(t)\end{array}\right]
=\left[\begin{array}{cc}\cos(\sqrt{a}t) & \sin(\sqrt{a}t)\\-\sin(\sqrt{a}t) & \cos(\sqrt{a}t)\end{array}\right] \left[\begin{array}{c}\bx_c(0) \\ \bu_c(0)\end{array}\right]
\fi
 
\begin{remark}[{\bf blow-up as $a \ll 1$}]
We note in passing  that \eqref{cEinfty} does not recover the velocity alignment in the potential-free case due to the blow-up of  $\displaystyle \Ckitty 
={\mathcal O}(1/{a})$
as $a\rightarrow 0$. The growing bound  is due to  the proof in which we estimate the momentum $\phi*(\rho \bu)$ as a source term by using $L^2$ exponential decay in theorem \ref{thm_flocking1}: yet, the $L^2$-decay rate $\lambda(a)$  deteriorates as $a \rightarrow 0$, and  the effect of an increasing  source term leads to the blow-up of $\Ckitty$. Indeed, it is known that the unconditional velocity alignment in the potential-free case is restricted to the `fat-tails' \eqref{eq:fat}, hence                                                                                                                                                                                                               our approach for the thinner tails \eqref{eqs:admiss}
cannot apply uniformly in $1/a$.
\end{remark}

\section{Statement of main results --- flocking with general convex potentials}\label{sec:convex}
\subsection{General considerations} We now turn our attention to alignment dynamics  \eqref{eq} with more general strictly convex  potentials, \eqref{eq:uconvex}. The flocking results are more restricted. We begin with specifying the smaller class of admissible interaction kernels.

\begin{assumption}[{\bf Admissible kernels}]\label{uphicond}
We consider \eqref{eq} with interaction kernel $\phi$ such that
\begin{subequations}\label{eqs:uadmiss}
\begin{align}
&\mbox{(i)}   \quad \phi(r) \mbox{ is positive, decreasing and  bounded}: \ 
0< \phi(r) \le \phi(0):=\phi_+ < \infty; \label{uphicond0}\\
&\mbox{(ii)} \quad \phi(r) \mbox{ decays slow enough at infinity in the sense that} \
\limsup_{r\rightarrow\infty} r\phi(r) = \infty. \label{uphicond1}
\end{align}
\end{subequations}
\end{assumption}
\noindent
Notice that (\ref{uphicond1}) is only slightly more restrictive than the usual `fat-tail' assumption $\displaystyle \int_0^\infty \phi(r)\rd{r} = \infty$, which characterize unconditional  flocking in the case of potential-free alignment \cite{HT2008,HL2009}.

\medskip
We begin noting that the basic bookkeeping of  energy decay  \eqref{energy} still holds, 
\[
\frac{\rd}{\rd t} \Et(t) = -\frac{1}{2}\int\int \phi(|\bx-\by|)|\bu(\bx,t)-\bu(\by,t)|^2\rho(\bx,t)\rho(\by,t)\rd{\bx}\rd{\by}.
\]

\medskip
\paragraph{{\bf $\bullet$ Uniform bounds}} A  necessary main ingredient in the analysis of \eqref{eq} is the uniform bound of $\text{diam}(\textnormal{supp }\rho(\cdot,t))$, and the amplitude of velocity $\displaystyle \mathop{\max}_{\bx\in {\textnormal{supp } \rho}}|\bu(\bx,t)|$. Our next lemma shows that whenever one has a uniform bound of $|\bu(\bx,t)|+|\bx|$ for the \emph{restricted} class of lower-bounded $\phi$'s which scales like ${\mathcal O}(1/\min \phi)$, then it implies a uniform bound of 
$|\bu(\bx,t)|+|\bx|$ for the general class of admissible $\phi$'s \eqref{eqs:admiss}.

\begin{lemma}[{\bf The reduction to lower-bounded $\phi$'s}]\label{cor_reduce}
Consider \eqref{eq} with a with the \underline{restricted} class of lower-bounded $\phi$'s:
\begin{equation}\label{phiminus}
0< \phi_- \le \phi(r) \le \phi_+ < \infty.
\end{equation}
Assume that the solutions  $(\widetilde{\rho},\widetilde{\bu})$ associated with the restricted \eqref{eq},\eqref{phiminus},  satisfy the uniform bound (with constants   $C_\pm$ depending on $\FP, \phi_+,m_0$ and $\Etin$)
\begin{equation}\label{eq:ubound}
\max_{t\ge 0,\, \bx\in\textnormal{supp }\tilde{\rho}(\cdot,t)} (|\widetilde{\bu}(\bx,t)|+|\bx|) \le \max\left\{\Cplus\cdot\hspace*{-0.4cm}\max_{\quad \bx\in\textnormal{supp }\widetilde{\rho}_0} (|\widetilde{\bu}_0(\bx)| + |\bx|)\,,\frac{\Ccat}{\phi_-}\right\}.
\end{equation}
Then the following holds for solutions associated with a \underline{general}  admissible kernel $\phi$ \eqref{eqs:uadmiss}: if $(\rho,\bu)$ is a smooth solution of \eqref{eq}, then there exists $\alpha>0$ (depending on the initial data  $(\rho_0,\bu_0)$),  such that $(\rho,\bu)$ coincides with  the solution, $(\widetilde{\rho}_\alpha,\widetilde{\bu}_\alpha)$, associated with the lower-bounded  $\phi_\alpha(r) := \max\{\phi(r),\alpha\}$.
\end{lemma}
This means that if $\phi$ belongs to the general class of admissible kernels \eqref{eqs:uadmiss}, then we can  assume, without loss of generality, that $\phi$ coincides with the lower bound $\phi_\alpha$ and hence the uniform bound \eqref{eq:ubound} holds with $\phi_-=\alpha$. The justification of this reduction step is outlined below.

\begin{proof}[Proof of Lemma \ref{cor_reduce}]
By the condition (\ref{phicond1}), there exists $r_0$ such that
$r_0 \phi(r_0) \ge 2\Ccat$, 
and one could take large enough $r_0$ such that
\begin{equation}\label{r0large}
r_0\ge 2\Cplus\cdot \hspace*{-0.4cm}\max_{\quad \bx\in\textnormal{supp }\rin}\hspace*{-0.4cm}(|\buin(\bx)|+|\bx|).
\end{equation}
Let $\alpha=\phi(r_0)$. By assumption, (\ref{eq:ubound}) holds for the lower-bounded $\phi_\alpha$, so that
\begin{equation}
\max_{t\ge 0,\, \bx\in\textnormal{supp }\rho_\alpha(\cdot,t)} (|\bu_\alpha(\bx,t)|+|\bx|) \le \max\Big\{\Cplus\cdot\hspace*{-0.4cm}\max_{\quad \bx\in\textnormal{supp }\rin}\hspace*{-0.4cm} (|\buin(\bx)|+|\bx|), \frac{\Ccat}{\alpha}\Big\}
\end{equation}
where $(\rho_\alpha,\bu_\alpha)$ is the smooth solution of (\ref{eq}) with interaction kernel  $\phi_\alpha$, which we assume to exist. Therefore, for any $t\ge 0$ and any $\bx,\by \in \textnormal{supp }\rho_\alpha(\cdot,t)$, we have
\begin{equation}
|\bx-\by| \le |\bx|+|\by| \le 2\max\Big\{\Cplus\cdot \hspace*{-0.4cm}\max_{\quad\bx\in\textnormal{supp }\rin}\hspace*{-0.4cm} (|\buin(\bx)|+|\bx|), \frac{\Ccat}{\alpha}\Big\}
\end{equation}
By definition,
\begin{equation}
\frac{\Ccat}{\alpha} = \frac{\Ccat}{\phi(r_0)} \le \frac{r_0}{2}
\end{equation}
Together with (\ref{r0large}), we obtain $|\bx-\by| \le r_0$
for which,  by the monotonicity of $\phi$,
$\phi(|\bx-\by|) \ge \phi(r_0) = \alpha$. But for this  $\bx,\by$ which persist with  a ball of diameter $r_0$ we have $\phi(|\bx-\by|) = \phi_\alpha(|\bx-\by|)$
so the dynamics of $(\rho_\alpha,\bu_\alpha)$ coincides with $(\rho,\bu)$.
\end{proof}

\begin{remark}
For the special case
$\displaystyle 
\phi(r) = \frac{\phi_+}{(1+r^2)^{\beta/2}}$
with $\beta<1$, the proof of Corollary \ref{cor_reduce} shows that one could take
\begin{equation}
\alpha=\phi(r_0),\quad r_0 = \max\left\{4\left(\frac{\Ccat}{\phi_+}\right)^{\frac{1}{1-\beta}}, \, 2\Cplus\cdot \hspace*{-0.4cm}\max_{\quad \bx\in\textnormal{supp }\rin} \hspace*{-0.4cm} (|\buin(\bx)|+|\bx|)\right\}
\end{equation}
Therefore, the lower cut-off at $\alpha$, which depends on $\beta,m_0,\phi_+$ and the initial data, gets smaller when $\beta$ approaches 1.
\end{remark}

The following proposition asserts the uniform bounds (\ref{eq:ubound}) exist for the \emph{restrictive} class of kernels bounded from below, under very mild conditions on $\FP$.

\begin{proposition}\label{prop_infty2}
Assume the potential $\FP$ satisfies 
\begin{equation}\label{assu_Aa}
\frac{a}{2}|\bx|^2 \le \FP(\bx) \le \frac{A}{2}|\bx|^2,\quad a|\bx|\le |\nabla \FP(\bx)| \le A|\bx|,\quad \forall \bx\in\Omega, \quad 0<a\le A.
\end{equation}
 Consider the alignment system  \eqref{eq},\eqref{assu_Aa} with an interaction kernel which is assumed to be  bounded from below, \eqref{phiminus}.
Then there exist constants $C_\pm$, depending on  $\FP, \phi_+,m_0$ and  $\Etin$, such that \eqref{eq:ubound} holds.
\end{proposition}
\begin{remark}\label{rem:stability}
We note in passing that if $\FP$ is strictly convex potential satisfying \eqref{eq:uconvex}  then \eqref{assu_Aa} follows. Indeed, assuming without loss of generality, that $\FP$ has  a global minimum at the origin so that $\FP(0)=\nabla \FP(0)=0$, and expressing $\nabla \FP(\bx) = \int_0^1 \nabla^2 U(s\bx)\bx \rd{s}$ we find $|\nabla U(\bx)| \le \int_0^1 A|\bx| \rd{s} = A|\bx|$ while
strict convexity implies
\[
\bx\cdot \nabla \FP(\bx) = \int_0^1 \bx^\top \nabla^2 U(s\bx)\bx \rd{s} \ge a|\bx|^2 \quad \leadsto \quad |\nabla \FP(\bx)| \ge a|\bx|;
\]
moreover, expressing $U(\bx) = \int_0^1 \nabla \FP(s\bx)\cdot\bx \rd{s}$ we find
\[
\frac{a}{2}|\bx|^2 =  \int_0^1 \frac{1}{s} a|s\bx|^2 \rd{s}\leq \int_0^1 \frac{1}{s}\nabla\FP(s\bx)\cdot s\bx \rd{s}   \leq U(\bx) \le \int_0^1 A|s\bx|\cdot|\bx| \rd{s} = \frac{A}{2}|\bx|^2.
\]
Thus, the assumed bounds \eqref{assu_Aa} follow from \eqref{eq:uconvex}.  In fact, \eqref{assu_Aa} 
 allows more general scenarios than uniform convexity including, notably, more  complicated topography involving than one local minima. The flocking behavior of such scenarios are considerably more intricate, consult \cite{HS2018}.\newline 
 It is straightforward to generalize Proposition \ref{prop_infty2} to the case when \eqref{assu_Aa} only holds for sufficiently large $|\bx|$. We omit the details.
\end{remark}

\subsection{Flocking of smooth solutions with convex potentials}

From now on we will restrict attention to uniformly lower bounded kernels, so that $\phi$ satisfies \eqref{phiminus}, $0<\phi_-\leq \phi(\bx)\leq \phi_+$. The reduction Lemma \ref{cor_reduce} tells us that the results will automatically apply to the class of all admissible kernels which satisfy \eqref{eqs:admiss}.
We develop a hypocoercivity argument, different from the one used in the quadratic case,  which gives  the following $L^2$-flocking estimate with algebraic decay rate.
\begin{theorem}[{\bf Flocking with  $L^2$-algebraic decay}]\label{thm_flocking3}
Consider the system \eqref{eq} with uniformly convex potential \eqref{eq:uconvex}, $0< aI_{d\times d} \leq \nabla^2\FP(\bx)\leq AI_{d\times d}$ and 
with a $C^1$ admissible  interaction kernel $\phi$,  \eqref{eqs:uadmiss}.
Assume, in addition, that $\phi$ satisfies the linear stability condition
\begin{equation}\label{Kcond1}
m_0\phi(0)>  \frac{A}{\sqrt{a}}.
\end{equation}
Let $(\rho,\bu)$ be a global smooth solution subject to compactly support $\rin$.
 Then there holds flocking at algebraic rate in both velocity and position, namely, there exist a constant $C$ (with increasing dependence on $|\phi'|_\infty$) such that 
\begin{equation}\label{flocking3}
\cEt(t):= \int\int (|\bu(\bx)-\bu(\by)|^2 + a|\bx-\by|^2)\rho(\bx)\rho(\by) \rd{\bx}\rd{\by} \le \frac{C}{\sqrt{1+t}}\cEtin.
\end{equation}
\end{theorem}
The proof of Theorem \ref{thm_flocking3} involves three ingredients. First, from the total energy estimate, we show that when $t$ is large enough, most of the agents almost concentrate as a Dirac mass, traveling at almost the same velocity. Second, for such a concentrated state, one can replace $\phi$ by the \emph{constant} kernel $\phi(0)$ without affecting the dynamics too much, which in turn implies that the agents near the Dirac mass will be attracted to it, consult theorem \ref{thm_flocking2} below. Third, this gives some monotonicity of the energy dissipation rate, which in turn gives (\ref{flocking3}).

The $L^\infty$ counterpart of Theorem \ref{thm_flocking3} is still open. If one could obtain an $L^\infty$ flocking estimate, then it might be possible to have flocking estimates for $\phi$ with thinner tails, similar to what was done in sec. \ref{sec:quadratic}.

The origin of  the stability condition \eqref{Kcond1} can be traced to the case of a \emph{constant kernel}, $\phi$, where the algebraic convergence stated in theorem \ref{thm_flocking3} is in fact improved  to exponential rate.

\begin{theorem}[{\bf Flocking with  $L^2$-exponential decay-- constant $\phi$}]\label{thm_flocking2}
Let $(\rho,\bu)$ subject to compactly supported $\rin$ be a global smooth solution of \eqref{eq} with uniformly convex potential \eqref{eq:uconvex}, $0< aI_{d\times d} \leq \nabla^2\FP(\bx)\leq AI_{d\times d}$, and assume that the interaction kernel 
$\phi$ is \underline{constant} satisfying
\begin{equation}\label{Kcond0}
m_0\phi >  \frac{A}{\sqrt{a}}
\end{equation}
Then it undergoes unconditional flocking at exponential rate in both velocity and position: there exist $\lambda>0$ and $C>0$ depend on $a,A ,m_0\phi$ such that
\begin{equation}\label{flocking}
\cEt(t) \le C\cdot \cEtin \cdot e^{-\lambda t}.
\end{equation}
\end{theorem}

\begin{remark}
One may wonder about the necessity of the stability condition \eqref{Kcond1}. In fact, already in the simplest case of a constant $\phi$ where the  Cucker-Smale \eqref{CS_par1} is reduced to   
\begin{equation}
\left\{\begin{split}
& \dot{\bx}_i = \bv_i \\
& \dot{\bv}_i = \phi\cdot(\bar{\bv}-\bv_i) - \nabla\FP(\bx_i)
\end{split}\right.\qquad \bar{\bv}:=\frac{1}{N}\sum_j \bv_j,
\end{equation}
one may encounter 'orbital instability', where arbitrarily small initial fluctuations 
$|\bx_i(0)-\bx_j(0)|+|\bv_i(0)-\bv_j(0)|$ subject to 1d \underline{non-convex} potential may grow to be ${\mathcal O}(1)$ at some time, \cite{HS2018}. The stability condition \eqref{Kcond1} guarantees, in the case of convex potentials, strong enough alignment that prevents scattering and eventual flocking. The question of the precise necessary stability condition  vis a vis convexity remains open. 
\end{remark}

\section{Existence of global smooth solutions}\label{sec:existence}
According to proposition \ref{prop_infty2}, convex potentials guarantee that 
the reduction lemma  \ref{cor_reduce} holds, hence we can focus our attention, without loss of generality, on lower-bounded kernels such that $\phi_-=\min \phi(\cdot)>0$.  

\subsection{Existence of 1D solutions with general convex potentials}
We begin with  one-dimension (for which $\bu,\bx$ are scalars, written as $u,x$). The 1D setup is covered in the next two theorems,  where we\newline
 (i)  guarantee the existence of global smooth solution for a class of sub-critical initial configurations; and\newline
  (ii) guarantee a finite time blow-up for a class of super-critical initial configurations.\newline
  
\begin{theorem}[{\bf Global smooth solutions --- 1D problem}]\label{thm_1dsmooth}
Let the space dimension $d=1$. Assume $\FP''$ is bounded 
\begin{equation}\label{1d_assu1}
a\leq \FP''(x)\le A,\quad \forall x\in\Omega
\end{equation}
with $A$ being a constant satisfying
\begin{equation}\label{1d_assu2}
A < \frac{(m_0\phi_-)^2}{4}.
\end{equation}
Further assume that
\begin{equation}\label{1d_assu3}
\max_{x\in\textnormal{supp }\rin}(\partial_x \uin(x) + (\phi*\rin)(x)) >\frac{m_0\phi_-}{2} - \sqrt{\frac{(m_0\phi_-)^2}{4} - A}
\end{equation}
then \eqref{eq} admits global smooth solution.
\end{theorem}
Observe that the statement of theorem \ref{thm_1dsmooth} is independent of the lower-bound $a$, whether positive of negative: its only role  enters in the upper-bound of 
\[
\max u_x(\cdot,t) \lesssim \max\Big\{c_0(\max_x u'_0, m_0,\phi_+),\sqrt{\max\{0,-2a\}}\Big\}.
\]
\begin{theorem}[{\bf Finite-time blow-up --- 1D problem}]\label{thm_1dblowup}
Assume $\displaystyle \FP''(x)\ge a,\  \forall x\in\Omega$. The 1D problem \eqref{eq} admits finite-time blow-up under the following circumstances.

\noindent
{\rm (i)}  If $a$ is  large enough so that 
\begin{equation}\label{assuB_1}
a > \frac{(m_0\phi_+)^2}{4}, 
\end{equation}
 then there is unconditional blowup: $\partial_x u$ blows up to $-\infty$ in finite time for any initial data.\newline
Otherwise, blow-up occurs  if the initial data is super-critical in one of the following two configurations:

\smallskip\noindent
{\rm (ii)} If $a>0$ is not large enough for \eqref{assuB_1} to hold\footnote{Notice that in this condition the RHS in \eqref{assuB_2} is positive.}, then $\partial_x u$ blows up to $-\infty$ in finite time if there exists $x\in \Omega$ such that
\begin{equation}\label{assuB_2}
\partial_x \uin(x) + (\phi*\rin)(x) < \frac{m_0\phi_+}{2} - \sqrt{\frac{(m_0\phi_+)^2}{4} - a}.
\end{equation}

\smallskip\noindent
{\rm (iii)} If $a\le 0$, then $\partial_x u$ blows up to $-\infty$ in finite time if there exists $x\in \Omega$ such that\footnote{Notice that in this condition the RHS of \eqref{assuB_3} is negative.}
\begin{equation}\label{assuB_3}
\partial_x \uin(x) + (\phi*\rin)(x) < \frac{m_0\phi_-}{2} - \sqrt{\frac{(m_0\phi_-)^2}{4} - a}.
\end{equation} 

\end{theorem}

Note that in the potential-free case $\FP=0$,  theorems \ref{thm_1dsmooth} and\ref{thm_1dblowup} amount to  the  sharp threshold condition  $\partial_x \uin(x) + (\phi*\rin)(x)\geq 0$ which is  necessary and sufficient for global 1D regularity, see~\cite{CCTT2016,ST2017a}. 
When the external potential $\FP$ is added, these theorems indicate  that convex $\FP$ enhances the scenario of blowup in \eqref{eq}, while concave $\FP$'s makes more restrictive scenarios for possible blow up. In other words, {\it the size of $\FP''$} determines the influence of the external potential on the threshold for the existence of global smooth solution. 

It is also interesting to see that the flocking phenomena is {\it not relevant} for the existence of global smooth solution. In fact, (\ref{1d_assu1}) does not require $\FP$ to be confining, i.e., $\lim_{|x|\rightarrow\infty} \FP(x) = \infty$. Even if $\FP$ is confining, it may happen that flocking phenomena do not happen at a rate which is uniform in initial data, see the 'orbital instability' examples in~\cite{HS2018}. All these complications do not affect the existence of global smooth solutions at all.

\subsection{Existence of 2D solutions with quadratic potentials}
We state our results on the critical thresholds for the existence of global smooth solution, for two space dimensions, for quadratic potentials.  
\begin{theorem}[{\bf Global smooth solutions with 2D quadratic potential}]\label{thm_2dsmooth1}
Consider the two-dimensional system \eqref{eq:quadratic} subject to initial data $(\rho_0,\bu_0)$. 
Let $(\eta_S)_0$ denote the spectral gap -- the difference between the two eigenvalues of the symmetric matrix $\nabla_S\bu_0:=\nicefrac{1}{2}(\nabla \bu_0 + (\nabla \bu_0)^\top)$. 
Assume that the initial data are sub-critical in the sense that the following holds
(in terms of  $\lambda$  given in \eqref{lambda} and $|\phi'|_\infty$)
\begin{align}
c_1^2 := m_0^2\phi_-^2 - \left(\max_{\bx\in\textnormal{supp\,}\rin}|(\eta_S)_0(\bx)| + \Cmeow \cdot\sqrt{\cEiin}\right)^2 \!\!- \!\!4a &> 0, \quad \Cmeow  := \frac{64}{\lambda}m_0|\phi'|_\infty\sqrt{\Ckitty } \label{c1}\\
\max_{\bx\in\textnormal{supp\,}\rin}(\nabla \cdot \buin(\bx) + (\phi*\rin)(\bx)) &\ge 0. \label{c1_cond1}
\end{align}
Then \eqref{eq} admits global smooth solution.
\end{theorem}

This result can be viewed as a generalization of the main result of~\cite{HeT2017}. Compared to the latter, besides the pointwise smallness requirements for $\eta_S$, the $L^\infty$ variation of $\bu$, and the quantity $\nabla \cdot \buin + (\phi*\rin)$, we also require the smallness of the $L^\infty$ variation in $\bx$, see (\ref{cEinfty}), the definition of $\cEi$. This is because the effect of the external potential may convert variation in $\bx$ into variation in $\bu$ of the same order after some time.

For $a\rightarrow 0$, one has $\Cmeow = {\mathcal O}(a^{-3/2})$, and for $a\rightarrow\infty$, one has $\Cmeow = {\mathcal O}(1)$. Therefore, the condition (\ref{c1}) cannot hold if $a$ is either too small (the $\Cmeow$ term will blow up) or too large (the $4a$ term will blow up). Intuitively speaking, the reason for blow-up in the first case is that one does not have a good flocking estimate, and thus the velocity variation may affect the dynamics of $\nabla \bu$ in an uncontrollable way. The reason in the second case is similar to the 1d case: a 'very convex' potential tends to induce blow-up directly. Therefore, in order to guarantee the existence of two-dimensional global smooth solution, one first needs $m_0\phi_-$ large enough, and then taking moderately size $a$ will satisfy (\ref{c1}), if the initial data is well-chosen ($\eta_S$, $\cEi$ not too large and  $\nabla \cdot \buin + (\phi*\rin)$ non-negative). 

\subsection{Existence of 2D solutions with general convex potentials}
For the existence of global smooth solution for general external potentials, one difficulty is as follows: a critical property of the quadratic potential used in the proof of Theorem \ref{thm_2dsmooth1} is that it has no effect on the dynamics of $\eta_S$ (which is a crucial ingredient of the proof), since the Hessian $\nabla^2 \FP$ is constant multiple of the identity matrix. However, this is not true in general, and the effect of the external potential on $\eta_S$ can be as large as the distance between the two eigenvalues of $\nabla^2 \FP$. Another difficulty is that for many cases of $\FP$ we do not have a large time flocking estimate, and the contribution from the variation of $\bu$ to the dynamics of $\eta_S$ may accumulate over time. Interestingly, we discover that both issues can be resolved by requiring slightly strengthening the critical threshold (as in \cite{TW2008}): instead of requiring the quantity $\nabla \cdot \buin + \phi*\rin$  nonnegative, we require it to have a positive lower bound. (In fact, one expects the second difficulty not to be essential, since the 1d case suggests that flocking estimates should not be a necessary ingredient for the existence of global smooth solution.)

\medskip

\begin{theorem}[{\bf Global 2D smooth solutions with  convex potential}]\label{thm_2dsmooth2}
Consider the two-dimensional system \eqref{eq} subject to initial data $(\rho_0,\bu_0)$, with external potential $\FP$ being sub-quadratic:
\begin{equation}\label{assu_A}
|\nabla^2\FP(\bx)|\le A.
\end{equation}
Assume  the apriori uniform bound on the velocity field holds,
\begin{equation}\label{umax}
\max_{t\ge 0,\,\bx\in\textnormal{supp }\rho(\cdot,t)}|\bu(\bx,t)| \le u_{max} <\infty.
\end{equation}
If the initial data, $(\rho_0,\bu_0)$, are sub-critical in the sense that the following holds 
\begin{align}
\Cmi  := 8|\phi'|_{\infty} m_0 u_{max} + 2A & < \frac{m_0^2\phi_-^2}{2} - 2A =: \CA, \label{Cmi}\\
\max_{\bx\in\textnormal{supp }\rin} |(\eta_S)_0(\bx)| & \le \sqrt{\CA + \sqrt{\CA^2-\Cmi ^2}}, \label{etaS_cond2}\\
\max_{\bx\in\textnormal{supp }\rin} (\nabla \cdot \buin(\bx) + (\phi*\rin)(\bx)) &> \sqrt{\CA - \sqrt{\CA^2-\Cmi ^2}},\label{c1_cond2}
\end{align}
then \eqref{eq} admits global smooth solution.
\end{theorem}
Notice that Proposition \ref{prop_infty2} already gives an a priori estimate
\begin{equation}
u_{max} =  \max\left\{C_+\cdot\max_{\bx\in\textnormal{supp }\rin} (|\buin(\bx)|+|\bx|),\frac{\Ccat}{\phi_-}\right\}
\end{equation}
for a general class of external potentials, including those satisfying (\ref{eq:uconvex}) (with the further assumption that the unique global minimum of $\FP$ is $\FP(0) = 0$, without loss of generality). Also, Theorem \ref{thm_2dsmooth2} also applies to the cases when other a priori estimates of $|\bu|$ are available.

\section{Proof of main results --- hypocoercivity bounds}

\subsection{Quadartic potentials}
We prove theorems \ref{thm_flocking1} and \ref{prop_infty1}, making use of the uniform lower-bound of $\phi(r) \geq \phi_-$ in \eqref{eq:phimin}.
\begin{proof}[Proof of theorem \ref{thm_flocking1}]
Since the fluctuations functional $\cEt(\rho,\bu)$ in (\ref{flocking}) satisfies $\cEt(\rho,\bu)=\cEt(\hat{\rho},\hat{\bu})$,  it suffices to study (\ref{eq:quadratic}) with 
$(\bx_c(0)=0,\bu_c(0))=(0,0) \leadsto (\bx_c(t),\bu_c(t))\equiv (0,0)$, for which the fluctuations coincide with  (multiple of)  the energy
\begin{equation}\label{eq:delvsE}
\cEt(t)  = 4m_0 \int \Big(\frac{1}{2}|\bu(\bx,t)|^2+\frac{a}{2}|\bx|^2\Big)\rho(\bx,t)\rd{\bx}.
\end{equation}
As before, the energy decay is dictated by the minimal value $\displaystyle  \mathop{\min}_{\bx,\by\in\textnormal{supp\,}\rho(\cdot,t)}\phi(|\bx-\by|)\geq \phi_-:= \phi( \sqrt{8R_0/a})$,
\begin{equation}\label{eq:est1}
\begin{split}
\partial_t \int \Big(\frac{1}{2}|\bu(\bx,t)|^2&+\frac{a}{2}|\bx|^2\Big) \rho(\bx,t)\rd{\bx} 
\!=\!  -\frac{1}{2}\int\int\phi(\bx-\by) |\bu(\by)\!-\!\bu(\bx)|^2\rho(\bx)\rho(\by)\rd{\bx}\rd{\by} \\
\le & -\frac{\phi_-}{2}\int\int |\bu(\by)-\bu(\bx)|^2\rho(\bx)\rho(\by)\rd{\bx}\rd{\by} 
=  -m_0\phi_-\int |\bu|^2\rho\rd{\bx}.
\end{split}
\end{equation}
Then we compute the cross term
\[
\begin{split}
& \partial_t \int \bu(\bx,t)\cdot\bx \rho(\bx,t)\rd{\bx} \\
= & - \int (\bu(\bx,t)\cdot\bx)\nabla\cdot(\rho\bu)\rd{\bx} + \int \bx\cdot\left(-\bu\cdot\nabla\bu + \int \phi(\bx-\by)(\bu(\by)-\bu(\bx))\rho(\by)\rd{\by} - a\bx\right)\rho\rd{\bx} \\
= &   - a \int |\bx|^2\rho\rd{\bx} + \int |\bu|^2 \rho \rd{\bx} + \int \int \phi(\bx-\by) \bx\cdot(\bu(\by)-\bu(\bx))\rho(\bx)\rho(\by) \rd{\bx}\rd{\by} \\
\le &   - a\int |\bx|^2\rho\rd{\bx} + \int |\bu|^2 \rho \rd{\bx} + \frac{\phi_+}{2}\int \int \Big( \frac{a}{m_0\phi_+}|\bx|^2+\frac{m_0\phi_+}{a}|\bu(\by)-\bu(\bx)|^2\Big)\rho(\bx)\rho(\by) \rd{\bx}\rd{\by} \\
= &   - \frac{a}{2}\int |\bx|^2\rho\rd{\bx} + \Big(1+\frac{m_0^2\phi_+^2}{a}\Big)\int |\bu|^2\rho\rd{\bx} \\
\end{split}
\]
Adding a $\mixed$-multiple of this cross term --- $\mixed$ is yet to be determined,  we conclude that 
\begin{equation}\label{hypo1}
\begin{split}
 \partial_t  \int \Big(\frac{1}{2}|\bu(\bx,t)|^2&+\frac{a}{2}|\bx|^2 + 2\mixed\bu(\bx,t)\cdot\bx\Big) \rho(\bx,t)\rd{\bx} \\
\le & -\Big(m_0\phi_--2\mixed\Big(1+\frac{m_0^2\phi_+^2}{a}\Big)\Big) \int |\bu|^2\rho\rd{\bx}  - 2\mixed \int \frac{a}{2} |\bx|^2\rho\rd{\bx}.
\end{split}
\end{equation}
which means the LHS is a Lyapunov functional if $\mixed>0$ is small enough: in fact, we set
\be\label{eq:mixed1}
\mixed = \frac{1}{2}\min\left\{ \frac{m_0\phi_-}{(1 + \frac{m_0^2\phi_+^2}{a}) + \frac{1}{2}}, \frac{\sqrt{a}}{2} \right\},
\ee
to conclude that the Lyapunov functional
\begin{equation}\label{V0}
V(t) := \int \Big(\frac{1}{2}|\bu(\bx,t)|^2+\frac{a}{2}|\bx|^2 + 2\mixed\bu(\bx,t)\cdot\bx\Big) \rho(\bx,t)\rd{\bx},
\end{equation}
admits the decay bound $\displaystyle 
\frac{\rd}{\rd t}V(t) \le -\mixed \int( |\bu|^2+ a|\bx|^2)\rho\rd{\bx}$.
Noting that  this modified Lyapunov functional
 is comparable to the energy functional (recall  $2\mixed\le \nicefrac{\sqrt{a}}{2}$)
\[
\frac{\cEt}{4m_0}=\frac{1}{2}\int ( |\bu|^2+ a|\bx|^2)\rho\rd{\bx} \le V(t) \le \int( |\bu|^2+ a|\bx|^2)\rho\rd{\bx} =\frac{\cEt}{2m_0},
\]
we conclude its \emph{dissipativity} $\displaystyle V'(t) \le -\mixed V(t)$
which in turn proves the $L^2$-flocking bound \eqref{flocking},
$\displaystyle \frac{\cEt(t)}{4m_0} \le V(t) \le \frac{\cEtin}{2m_0}e^{-\mixed t}$. 
\end{proof}

\begin{proof}[Proof of theorem \ref{prop_infty1}]
We define the perturbed energy functional  
\begin{equation}
F_1(\bx,t) := \frac{1}{2}|\bu(\bx,t)|^2+\frac{a}{2}|\bx|^2 + 2\mixed_1\bu(\bx,t)\cdot\bx
\end{equation}
where $\mixed_1>0$ is yet to be determined. Then we compute the derivative of $F_1$ along characteristics:
\begin{equation}\label{hypo2}
\begin{split}
F'_1 = & \partial_t F_1 + \bu\cdot\nabla F_1 \\
= & (\bu+2\mixed_1\bx)\cdot\left(-\bu\cdot\nabla\bu + \int \phi(\bx-\by)(\bu(\by)-\bu(\bx))\rho(\by)\rd{\by} - a\bx\right) \\ & + \bu\cdot(\bu\cdot\nabla\bu)  + a\bu\cdot\bx + 2\mixed_1|\bu|^2 + 2\mixed_1\bx\cdot(\bu\cdot\nabla\bu) \\
= & -2\mixed_1 a|\bx|^2 + (\bu+2\mixed_1\bx)\cdot\left(  \int \phi(\bx-\by)(\bu(\by)-\bu(\bx))\rho(\by)\rd{\by}\right)  + 2\mixed_1|\bu|^2   \\
= & -\!2\mixed_1 a|\bx|^2 \!-\!(\phi*\rho)|\bu|^2 \!+\!\bu\cdot(\phi*(\rho\bu))\!+\! 2\mixed_1\bx\cdot((\phi*(\rho\bu))-(\phi*\rho)\bu)\!+\!2\mixed_1|\bu|^2.
\end{split}
\end{equation}
We bound the convolution terms of the right of \eqref{hypo2}: by  \eqref{grow} we have\newline $m_0\phi_- \le (\phi*\rho)(\bx) \le m_0\phi_+$; further, by \eqref{eq:delvsE} $\cEt(t) > 4m_0 E_k(t)$ and the exponential decay of $L^2$-Lyapunov functional, \eqref{flocking},  imply 
\begin{equation*}
\begin{split}
|(\phi*(\rho\bu))(\bx)| = &  \left| \int \phi(\bx-\by)\bu(\by)\rho(\by)\rd{\by} \right|   \\ \le &  \phi_+\int |\bu(\by)|\rho(\by)\rd{\by} \le \phi_+ \sqrt{m_0} \left(\int |\bu|^2\rho\rd{\by}\right)^{1/2} \le  \phi_+ \sqrt{m_0} \frac{\sqrt{2\cEtin}}{\sqrt{2m_0}}e^{-\mixed t/2}.
\end{split}
\end{equation*}
We conclude that the perturbed energy functional  $F_1$ does not exceed
\[
\begin{split}
F_1' \le & -2\mixed_1 a|\bx|^2 -m_0\phi_-|\bu|^2 + \Big(\frac{m_0\phi_-}{2}|\bu|^2+ \frac{\phi_+^2}{2m_0\phi_-}  \cEtin \cdot e^{-\mixed t}\Big) \\ & + \Big(\frac{\mixed_1 a}{2}|\bx|^2+ \frac{2\mixed_1\phi_+^2}{a} \cEtin \cdot e^{-\mixed t}\Big)  + 2\mixed_1m_0\phi_+\Big(\frac{a}{4m_0\phi_+}|\bx|^2 + \frac{m_0\phi_+}{a}|\bu|^2\Big) + 2\mixed_1|\bu|^2   \\
\le & -\mixed_1 a|\bx|^2- \Big(\frac{m_0\phi_-}{2} - 2\mixed_1\big(1+\frac{m_0^2\phi_+^2}{a}\big)\Big)|\bu|^2
 + C_0 \cdot \cEtin\cdot e^{-\mixed t}
\end{split}
\]
with 
\begin{equation}\label{eq:Ckitty}
C_0 = \Big(\frac{1}{2m_0\phi_-}+\frac{2\mixed_1}{a}\Big)\phi_+^2.  
\end{equation}
 Therefore, by choosing $\mixed_1$ as 
\begin{equation}\label{eq:mixed1}
\mixed_1 := \frac{1}{4}\min\Big\{ \frac{m_0\phi_-}{(1 + \frac{m_0^2\phi_+^2+1}{a}) + \frac{1}{4}}, \frac{\sqrt{a}}{2} \Big\} \ge \frac{\mixed}{2},
\end{equation}
one has
\[
F'_1(t) \le -\frac{\mixed}{2} (a|\bx|^2 + |\bu|^2) + C_0 \cdot \cEtin\cdot e^{-\mixed t} \le -\frac{\mixed}{2} F_1(t) + C_0 \cdot \cEtin\cdot e^{-\mixed t},
\]
with  the explicit bound $F_1(t) \leq e^{-\mixed t/2}\left(F_1(0) + {2C_0 \cdot \cEtin/\mixed } \right)$.
 Finally,   since $\displaystyle \mathop{\max}_{\bx\in\textnormal{supp\,}\rho(\cdot,t)}F_1(\bx,t)$ is comparable with  $\cEi$, namely $\displaystyle \frac{1}{8}\cEi \le F_1 \le \frac{1}{2}\cEi$ and $\displaystyle \cEt \le m_0^2\cdot \cEi$, the result \eqref{cEinfty} follows with
 $\Ckitty = 4(1+4C_0m_0^2/\mixed)$.
 \end{proof}

\subsection{General convex potentials}
We begin with the proof of Proposition \ref{prop_infty2}, which confirms the  the uniform bound  $|\bu|+|\bx|$ in terms of ${\mathcal O}(1/\phi_-)$. The main idea is to study the evolution of the particle energy $\frac{1}{2}|\bu(\bx,t)|^2+ \FP(\bx)$ along characteristics, and conduct hypocoercivity arguments to handle the possible increment of the particle energy due to the Cucker-Smale interaction.

\begin{proof}[Proof of Proposition \ref{prop_infty2}]
We define 
\begin{equation}
F(\bx,t) = \frac{1}{2}|\bu(\bx,t)|^2+ \FP(\bx) + c\bu(\bx,t)\cdot\nabla \FP(\bx)
\end{equation}
with $c>0$ being small, to be chosen. Then it follows from the assumptions on $\FP$ that 
\begin{equation}\label{Fbound}
\begin{split}
F - \frac{1}{4} |\bu|^2 - \frac{a}{4}|\bx|^2 = &\frac{1}{4}|\bu|^2+ (\FP(\bx)-\frac{a}{4}|\bx|^2) + c\bu(\bx,t)\cdot\nabla \FP(\bx)\\
  \ge & \frac{1}{4} |\bu|^2 +\frac{a}{4}|\bx|^2 - \frac{c}{2}(\frac{1}{4c}|\bu|^2 + 4c|\nabla\FP(\bx)|^2) \\
\ge & \frac{1}{8} |\bu|^2 + \frac{a}{4}|\bx|^2 - 2c^2 A^2|\bx|^2 \ge 0.
\end{split}
\end{equation}
Now fix $\displaystyle c \le \sqrt{\frac{a}{8A^2}}$. 
Then we compute the derivative of $F$ along characteristics:
\begin{align}
F' = & \partial_t F + \bu\cdot\nabla F \nonumber \\
= & (\bu+c\nabla \FP(\bx))\cdot\left(-\bu\cdot\nabla\bu + \int \phi(\bx-\by)(\bu(\by)-\bu(\bx))\rho(\by)\rd{\by} - \nabla \FP(\bx)\right) \nonumber \\ 
& + \bu\cdot(\bu\cdot\nabla\bu)  + \bu\cdot\nabla\FP(\bx) + c\bu^\top\nabla^2\FP(\bx)\bu + c\nabla\FP(\bx)\cdot(\bu\cdot\nabla\bu) \label{eq:here}\\
= & -c|\nabla\FP(\bx)|^2 + (\bu+c\nabla\FP(\bx))\cdot\left(  \int \phi(\bx-\by)(\bu(\by)-\bu(\bx))\rho(\by)\rd{\by}\right)+ c\bu^\top\nabla^2\FP(\bx)\bu   \nonumber\\
= & -c|\nabla\FP(\bx)|^2 -(\phi*\rho)|\bu|^2 + \bu\cdot(\phi*(\rho\bu)) + c\nabla\FP(\bx)\cdot((\phi*(\rho\bu))-(\phi*\rho)\bu) \nonumber\\
& + c\bu^\top\nabla^2\FP(\bx)\bu  \nonumber 
\end{align}
Noticing that $m_0\phi_- \le (\phi*\rho)(\bx) \le m_0\phi_+$, the convolution term on the right of \eqref{eq:here} can be upper-bounded in terms of the dissipating energy $\Et(t)$ in \eqref{energy} 
\begin{equation*}
\begin{split}
|(\phi*(\rho\bu))(\bx)| = &  \left| \int \phi(\bx-\by)\bu(\by)\rho(\by)\rd{\by} \right| \le \phi_+\int |\bu(\by)|\rho(\by)\rd{\by} \\ \le & \phi_+\int |\bu(\by)|\rho(\by)\rd{\by} \le \phi_+ m_0^{1/2} \left(\int |\bu|^2\rho\rd{\by}\right)^{1/2} \le 2\phi_+ m_0^{1/2} \Et^{1/2}(0) ,\quad \forall \bx.
\end{split}
\end{equation*}

Therefore
\begin{equation*}
\begin{split}
F' \le & -c|\nabla\FP(\bx)|^2 -m_0\phi_-|\bu|^2 + (\frac{m_0\phi_-}{2}|\bu|^2+ \frac{2}{m_0\phi_-}\phi_+^2 m_0 \Etin) \\
& + (\frac{c}{4}|\nabla\FP(\bx)|^2+ 4c\phi_+^2 m_0 \Etin)  + (cm_0\phi_+)(\frac{1}{4m_0\phi_+}|\nabla\FP(\bx)|^2 + m_0\phi_+|\bu|^2) + cA|\bu|^2  \\
\le & -\frac{c}{2}|\nabla\FP(\bx)|^2 - \big(\frac{m_0\phi_-}{2} - c(A+m_0^2\phi_+^2)\big)|\bu|^2 + C_0
\end{split}
\end{equation*}
with 
\begin{equation}\label{C_0}
C_0 = \Big(\frac{2}{m_0\phi_-}+4c\Big)\phi_+^2 m_0 \Etin
\end{equation}
 Therefore, by choosing 
\begin{equation}\label{c}
c = \min\left\{\frac{m_0\phi_-}{A+2(A+m_0^2\phi_+^2)},\sqrt{\frac{a}{8A^2}}\right\}
\end{equation}
 one has
\begin{equation}
F' \le -\frac{c}{2}(|\nabla\FP(\bx)|^2 + A|\bu|^2) + C_0
\end{equation}

Next we notice that
\begin{align*}
F & \le \frac{1}{2}|\bu|^2 + \frac{A}{2}|\bx|^2 + \frac{c}{2}(\frac{1}{c}|\bu|^2 + cA^2|\bx|^2) \\
& \le \max\Big\{1,\frac{1+c^2A}{2}\Big\}(|\bu|^2 + A|\bx|^2) = |\bu|^2 + A|\bx|^2
\end{align*}
and 
\begin{equation*}
|\nabla\FP(\bx)|^2 + A|\bu|^2 \ge \min\Big\{A, \frac{a^2}{A}\Big\}(|\bu|^2 + A|\bx|^2) = \frac{a^2}{A}(|\bu|^2 + A|\bx|^2)
\end{equation*}
This means that if 
\begin{equation}\label{C_F}
F(\bx,t) \ge \frac{2AC_0}{a^2c}:=C_F
\end{equation}
then $F'\le 0$. Thus $F$ cannot further increase (along characteristics) if it is larger than $C_F$. It is clear that $c=\mathcal{O}(\phi_-)$ and $C_0 = {\mathcal O}(1/\phi_-)$ for small $\phi_-$. Therefore $C_F = {\mathcal O}(1/\phi_-^2)$.

Therefore, by (\ref{Fbound}) we get
\begin{equation*}
\begin{split}
|\bu| + |\bx| \le & 2\Big(1+\frac{1}{\sqrt{a}}\Big)\sqrt{F} \le 2(1+\frac{1}{\sqrt{a}})\sqrt{\max\{C_F,\max_{\bx\in\textnormal{supp }\rin} F(\bx,0)\}}  \\ 
\le & 2\Big(1+\frac{1}{\sqrt{a}}\Big)\sqrt{\max\{C_F,\max_{\bx\in\textnormal{supp }\rin} |\buin(\bx)|^2 + A|\bx|^2\}} \\
\le & \max\Big\{\Cplus\cdot \hspace*{-0.3cm}\max_{\quad \bx\in\textnormal{supp }\rin}\hspace*{-0.4cm} (|\buin(\bx)| + |\bx|), 2(1+\frac{1}{\sqrt{a}})\sqrt{C_F},\Big\}, \quad \Cplus:=2\sqrt{A}(1+\frac{1}{\sqrt{a}}) \\
\end{split}
\end{equation*}
and the term $2(1+\frac{1}{\sqrt{a}})\sqrt{C_F}$ scales like ${\mathcal O}(1/\phi_-)$ for small $\phi_-$.
\end{proof}

When dealing with convex potential $\FP(\bx)=\frac{a}{2}|\bx|^2$ we used the fact that the mean location $\bx_c$ and mean velocity $\bu_c$ satisfies the closed system, \eqref{eqs:xcuc}, which enabled us to convert the measure of $L^2$-fluctuations into an energy-based functional.
In case of general convex potentials, however, the mean location $\bx_c$ and mean velocity $\bu_c$ do not satisfy a closed system and  therefore one cannot reduce the problem with $\bx_c=\bu_c=0$, for which $\cEt$ is equivalent to the total energy. Therefore one cannot using hypocoercivity on the energy estimate to obtain the decay of $\cEt$. Instead, we will construct a Lyapunov functional which is equivalent to $\cEt$ directly.
We begin with the case of a constant interaction kernel.
\begin{proof}[Proof of Theorem \ref{thm_flocking2}]
Recall that we assumed $\phi$ is constant. Denote
$K := m_0\phi$ so that the  convolution terms with $\phi$ amount to simple averaging, $(\phi*f)(\bx) =K \int f \rd{\bx}$. 
We will use the $\rho$-weighted  quantities  
\[
\langle f(\bx,\by),g(\bx,\by)\rangle := \int\int f(\bx,\by)\cdot g(\bx,\by)\rho(\bx)\rho(\by) \rd{\bx}\rd{\by},\quad |f(\bx,\by)|^2 := \langle f(\bx,\by),f(\bx,\by)\rangle
\]
for any scalar or vector functions $f,g$, where we suppress its dependence on $t$.

We compute the time derivative of the following quantity (where $\beta>0$ to be determined):
\begin{equation}
F(t) = \frac{K}{2}|\bx-\by|^2 + \langle \bx-\by,\bu(\bx)-\bu(\by) \rangle + \frac{\beta}{2}|\bu(\bx)-\bu(\by)|^2
\end{equation}
\begin{align}
\frac{\rd{F}}{\rd{t}} 
=  &\int\int \Big[ (\frac{K}{2}|\bx-\by|^2 + (\bx-\by)\cdot(\bu(\bx)-\bu(\by)) \\
& \qquad + \frac{\beta}{2}|\bu(\bx)-\bu(\by)|^2)(-\nabla_{\bx}\cdot(\rho(\bx)\bu(\bx))\rho(\by) - \nabla_{\by}\cdot(\rho(\by)\bu(\by))\rho(\bx)) \nonumber \\ 
&   \qquad +  (\bx-\by +\beta(\bu(\bx)-\bu(\by)))\cdot(-\bu(\bx)\cdot\nabla_{\bx}\bu(\bx)+\bu(\by)\cdot\nabla_{\by}\bu(\by) \nonumber \\
& \qquad - K\bu(\bx) + K\bu(\by) - \nabla\FP(\bx) + \nabla\FP(\by))\rho(\bx)\rho(\by) \Big]\rd{\bx}\rd{\by} \nonumber \\
= & \int\int \Big[ K(\bx-\by) + (\bu(\bx)-\bu(\by)) + \nabla_{\bx}\bu(\bx)(\bx-\by+\beta(\bu(\bx)-\bu(\by))))\cdot\bu(\bx) \label{K1eq0} \\ 
& \qquad + ( - K(\bx-\by) - (\bu(\bx)-\bu(\by)) - \nabla_{\by}\bu(\by)(\bx-\by+\beta(\bu(\bx)-\bu(\by))))\cdot\bu(\by)\nonumber \\
&  \quad + (\bx-\by +\beta(\bu(\bx)-\bu(\by)))\cdot(-\bu(\bx)\cdot\nabla_{\bx}\bu(\bx)+\bu(\by)\cdot\nabla_{\by}\bu(\by)  \nonumber \\ 
& \qquad - K\bu(\bx) + K\bu(\by) - \nabla\FP(\bx) + \nabla\FP(\by))\Big]\rho(\bx)\rho(\by) \rd{\bx}\rd{\by} \nonumber \\
= & \int\int \Big[   -(K\beta-1) |\bu(\bx)-\bu(\by)|^2 - (\bx-\by)\cdot(\nabla\FP(\bx)-\nabla\FP(\by)) \nonumber \\
&  \qquad - \beta(\bu(\bx)-\bu(\by))\cdot(\nabla\FP(\bx)-\nabla\FP(\by)) \Big]\rho(\bx)\rho(\by) \rd{\bx}\rd{\by}. \nonumber 
\end{align}

Notice that
\begin{equation}
(\bx-\by)\cdot(\nabla\FP(\bx)-\nabla\FP(\by)) = \int_0^1 (\bx-\by)^\top \nabla^2\FP((1-\theta)\by + \theta\bx) (\bx-\by) \rd{\theta} \ge a|\bx-\by|^2
\end{equation}
and similarly
\begin{equation}
|(\bu(\bx)-\bu(\by))\cdot(\nabla\FP(\bx)-\nabla\FP(\by))| \le A|\bu(\bx)-\bu(\by)|\cdot|\bx-\by|
\end{equation}
Then we obtain
\begin{equation}\label{Kbeta}
\begin{split}
(\ref{K1eq0}) \le & - (K\beta-1) |\bu(\bx)-\bu(\by)|^2  - a|\bx-\by|^2 + A\beta|\bu(\bx)-\bu(\by)|\cdot|\bx-\by|  \\
\end{split}
\end{equation}
We want to choose a $\beta$ such that the RHS of (\ref{Kbeta}), as a quadratic form, is negative-definite, i.e., 
its discriminant is
\begin{equation}
A^2\beta^2 - 4a(K\beta-1) = A^2\beta^2 - 4aK\beta + 4a < 0
\end{equation}
This is possible, since by (\ref{Kcond0})
$(4aK)^2 - 16A^2a = 16a(aK^2 - A^2) > 0$, 
and we can take
\begin{equation}
\beta := \frac{2aK}{A^2}
\end{equation}
and then 
\begin{equation}
\frac{\rd{F}}{\rd{t}} \le -\mu_1 ( |\bu(\bx)-\bu(\by)|^2  + a|\bx-\by|^2) = -\mu_1 \cEt
\end{equation}
for some $\mu_1>0$ (whose explicit form will be given in Remark \ref{rmk_mu1}).
With this choice of $\beta$, the discriminant of the LHS of (\ref{K1eq0}) is
\[
1^2 - 4\frac{K}{2}\frac{\beta}{2} = 1-\frac{2aK^2}{A^2} < 1-\frac{2aA^2}{aA^2} = -1
\]
and thus it is positive definite. One can estimate $F$ above and below by
$\mu_3 \cEt \le F \le \mu_2 \cEt$
for some $\mu_2>\mu_3>0$.
Therefore
$F(t) \le F(0)e^{-\frac{\mu_1}{\mu_2}}$
and then
\[
\cEt(t) \le \frac{1}{\mu_3}F(t) \le \frac{1}{\mu_3}F(0)e^{-\frac{\mu_1}{\mu_2}}  \le \frac{\mu_2}{\mu_3}\cEt(0)e^{-\frac{\mu_1}{\mu_2}} 
\]
\end{proof}

\begin{remark}
The key idea of the proof is the cancellation of the term $K(\bx-\by)\cdot(\bu(\bx)-\bu(\by))$ in (\ref{K1eq0}). For large $K$, this term is $O(K)$, while the two good terms are ${\mathcal O}(K)$ and ${\mathcal O}(1)$ respectively. If this term was not cancelled, then it could not be absorbed by the good terms. 

In fact, the positive/negative $K(\bx-\by)\cdot(\bu(\bx)-\bu(\by))$ terms are given by the time derivative of $\frac{K}{2}|\bx-\by|^2$ and  $\langle \bx-\by,\bu(\bx)-\bu(\by) \rangle$ respectively. Therefore, in the Lyapunov functional, one cannot change the coefficient ratio between a square term $|\bx-\by|^2$ and the cross term $\langle \bx-\by,\bu(\bx)-\bu(\by) \rangle$. This is an essential difference from the standard hypocoercivity theory (for which the cross term can be arbitrarily small).
\end{remark}

\begin{remark}\label{rmk_mu1}
One can obtain the explicit expression of  $\mu_1$  from \eqref{Kbeta} by letting the good terms absorb the bad term exactly, i.e., solving the quadratic equation
\[
(K\beta-1 - \mu_1)(a-a\mu_1) = \frac{A^2\beta^2}{4}
\]
yields $\displaystyle 
\mu_1 = \frac{aK^2}{A^2} - \sqrt{ \frac{a^2K^4}{A^4} -  \frac{aK^2}{A^2} + 1} > 0$; similarly, one obtains $\mu_{2,3}$ as
\[
\mu_{2,3} =\frac{1}{2a}\left( \frac{a^2K}{A^2}+\frac{K}{2} \pm \sqrt{ (\frac{a^2K}{A^2}+\frac{K}{2})^2 -  4a(\frac{aK^2}{2A^2} -\frac{1}{4})} \right) > 0.
\]
\end{remark}

To handle the case with non-constant $\phi$, we start with the following lemma:
\begin{lemma}\label{lem_sep}
With the same assumptions as Theorem \ref{thm_flocking3}, further assume  the apriori uniform bound on the velocity field:
\begin{equation}\label{umax1}
\max_{t\ge 0,\,\bx\in\textnormal{supp }\rho(\cdot,t)}(|\bu(\bx,t)|+|\bx|) \le u_{max} <\infty.
\end{equation}
Fix any $\epsilon_1$ small enough. Assume that at time $t_0$, one can write $\textnormal{supp }\rho(\cdot,t_0)$ into the disjoint union of two subsets:
\begin{equation}
\textnormal{supp }\rho(\cdot,t_0) = S_1\cup S_2,\quad S_1\cap S_2 = \emptyset
\end{equation}
which satisfies
\begin{equation}\label{eps1}
\int_{S_2} \rho(\bx,t_0)\rd{\bx} \le \eta \epsilon_1
\end{equation}
with $\eta>0$ depending on $\phi$, $\FP$, $u_{max}$ but independent of $\epsilon_1$,  and
\begin{equation}\label{eps2}
\cEi(t_0;S_1) := \sup_{\bx,\by\in S_1}( |\bu(\bx)-\bu(\by)|^2+a|\bx-\by|^2) \le \epsilon_1
\end{equation}
Let $S_1(t),S_2(t)$ be the image of $S_1,S_2$ under the characteristic flow map from $t_0$ to $t$. Then 
\begin{equation}\label{eps3}
\cEi(t;S_1(t)) \le \epsilon_1,\quad \forall t\ge t_0
\end{equation}
\end{lemma}
In this lemma, $S_1$ consists of the particles which are almost concentrated as a Dirac mass, and $S_2$ the other particles, which can be far away from the Dirac mass, but whose total mass is small. The lemma claims that the Dirac mass will not scatter around for all time. It can be viewed as a perturbative extension of the constant $\phi$ case, applied to the Dirac mass $S_1$.

Also notice that (\ref{eq:ubound}) gives (\ref{umax1}) with $u_{max}$ being the RHS of (\ref{eq:ubound}).

\begin{proof}
Define 
\begin{align*}
F(\bx,\by,t) &:= \frac{K}{2}|\bx-\by|^2 + (\bx-\by)\cdot(\bu(\bx,t)-\bu(\by,t)) + \frac{\beta}{2}|\bu(\bx,t)-\bu(\by,t)|^2,\\
 F_\infty(t;S) &= \max_{\bx,\by\in S} F(\bx,\by,t)
\end{align*}
where $K=m_0\phi(0)$, and the choice of $\beta$ is the same as the proof of Theorem \ref{thm_flocking2}, so that $F$ is a positive-definite quadratic form. Fix two characteristics $\bx(t)$ and $\by(t)$ with $\bx(t_0), \by(t_0)\in S_1$, and we compute the time derivative of $F$ along characteristics:
\[
\begin{split}
& \frac{\rd}{\rd{t}}F(\bx(t),\by(t),t) \\
& \qquad =  \partial_t F + \bu(\bx)\cdot \nabla_\bx F + \bu(\by)\cdot \nabla_\by F \\
& \qquad =  \left((\bx-\by) + \beta(\bu(\bx)-\bu(\by))\right)\cdot\Big( -\bu(\bx)\cdot\nabla_\bx \bu(\bx) +\bu(\by)\cdot\nabla_\bx \bu(\by)  \\
& \qquad \ \  \ + \int \phi(\bx-\bz)(\bu(\bz)-\bu(\bx))\rho(\bz)\rd{\bz} - \int \phi(\by-\bz)(\bu(\bz)-\bu(\by))\rho(\bz)\rd{\bz}  \Big) \\
& \qquad \ \ \ +  \bu(\bx)\cdot\left( K(\bx-\by) + (\bu(\bx)-\bu(\by)) + (\bx-\by)\cdot\nabla_\bx \bu(\bx) + \beta (\bu(\bx)-\bu(\by))\cdot\nabla_\bx \bu(\bx) \right) \\
& \qquad \ \  \ -  \bu(\by)\cdot\left( K(\bx-\by) + (\bu(\bx)-\bu(\by)) + (\bx-\by)\cdot\nabla_\by \bu(\by) + \beta (\bu(\bx)-\bu(\by))\cdot\nabla_\by \bu(\by) \right) \\
& \qquad =  -(K\beta-1) |\bu(\bx)-\bu(\by)|^2 \\
& \qquad \ \ \ - (\bx-\by)\cdot(\nabla\FP(\bx)-\nabla\FP(\by)) - \beta(\bu(\bx)-\bu(\by))\cdot(\nabla\FP(\bx)-\nabla\FP(\by)) \\
& \qquad \ \ \ + \big((\bx-\by) + \beta(\bu(\bx)-\bu(\by))\big)\cdot\Big(\int (\phi(\bx-\bz)-\phi(0))(\bu(\bz)-\bu(\bx))\rho(\bz)\rd{\bz} \\
& \qquad \ \ \ - \int (\phi(\by-\bz)-\phi(0))(\bu(\bz)-\bu(\by))\rho(\bz)\rd{\bz}  \Big).
\end{split}
\]
The first three terms are less than a negative definite quadratic form, as in the proof of Theorem \ref{thm_flocking2}. Now we handle the last term, which results from the fact that $\phi$ is not constant.

By the definition of $S_1(t)$, one has $\bx(t),\by(t)\in S_1(t)$ for all $t\ge t_0$.  If $\bz \in S_1(t)$, then $|\bx-\bz| \le \sqrt{\cEi(t;S_1(t))/a} \le C_1\sqrt{F_\infty(t;S_1(t))}$ for some constant $C_1$, since $F$ is comparable with $ |\bu(\bx)-\bu(\by)|^2+a|\bx-\by|^2$. Therefore
\begin{equation}
|\phi(\bx-\bz)-\phi(0)| \le |\phi'|_\infty C_1\sqrt{F_\infty(t;S_1(t))}
\end{equation}
It follows that
\begin{align*}
 \left|\left((\bx-\by) + \beta(\bu(\bx)-\bu(\by))\right)\!\cdot\! \int_{S_1(t)}\!\!\!\!\!(\phi(\bx-\bz)-\phi(0))(\bu(\bz)-\bu(\bx))\rho(\bz)\rd{\bz}\right|  \le  C_2 F_\infty(t;S_1(t))^{3/2}
\end{align*}
with $C_2 = (1/\sqrt{a}+\beta)m_0 |\phi'|_\infty C_1^3$.

If $\bz \in S_2(t)$, then we use the uniform bound (\ref{umax1}) to estimate $\bu(\bz)-\bu(\bx)$, and obtain
\begin{align*}
 \left|\left((\bx-\by) + \beta(\bu(\bx)-\bu(\by))\right)\!\cdot\! \int_{S_2(t)}\!\!\!\!\!(\phi(\bx-\bz)-\phi(0))(\bu(\bz)-\bu(\bx))\rho(\bz)\rd{\bz}\right|  \le  C_3 \eta \epsilon_1 F_\infty(t;S_1(t))^{1/2}
\end{align*}
with $C_3 = (1/\sqrt{a}+\beta)C_1\cdot 2\phi_+ \cdot 2 u_{max} $. Similar conclusions hold with $\bx$ and $\by$ exchanged.

Therefore we conclude that
\[
\frac{\rd}{\rd{t}}F(\bx(t),\by(t),t)  \le -\mu F(\bx(t),\by(t),t) +  C_2 F_\infty(t;S_1(t))^{3/2} + C_3 \eta \epsilon_1 F_\infty(t;S_1(t))^{1/2}
\]
with $\mu>0$ a constant. Taking $\bx(t),\by(t)$ as the characteristics where $\max_{\bx,\by\in S_1(t)} F(\bx,\by,t)$ is achieved, we obtain
\[
\frac{\rd{f}}{\rd{t}}  \le -\mu f +  C_2 f^{3/2} + C_3 \eta \epsilon_1 f^{1/2},\quad f(t) = F_\infty(t;S_1(t)).
\]
Now set $\displaystyle \eta = \frac{C_3}{C_2}$ and assume $\displaystyle \epsilon_1 \le \frac{\mu^2}{16C_2^2}$, then $\displaystyle \frac{\rd f}{\rd{t}} < 0$ whenever $f(t) = \epsilon_1$, and hence the bound  $f(t) < \epsilon_1$ persists in time. The conclusion of the theorem follows from the fact that $f$ and $\cEi(t;S_1(t))$ are comparable (up to adjust the upper bound $\epsilon_1$  by constant multiple).
\end{proof}
The next lemma guarantees the existence of a partition satisfying the assumptions of Lemma \ref{lem_sep}, in case the $L^2$ variation of velocity and location is small:
\begin{lemma}\label{lem_par}
With the same assumptions as in theorem \ref{thm_flocking3}, for any $\epsilon_1>0$, 
\begin{equation}
\cEt(t_0)<\frac{m_0\eta \epsilon_1^2}{2}
\end{equation}
implies the existence of a partition satisfying \eqref{eps1} and \eqref{eps2}.
\end{lemma}
\begin{proof}
Recall that $(\bx_c(t),\bu_c(t))$ denote the mean location and velocity \eqref{eq:xcuc}. Then
\begin{equation}
\begin{split}
& \int\int (|\bu(\bx)-\bu(\by)|^2+a|\bx-\by|^2)\rho(\bx)\rho(\by)\rd{\bx}\rd{\by} \\
= & \int\int (|(\bu(\bx)-\bu_c)-(\bu(\by)-\bu_c)|^2+a|(\bx-\bx_c)-(\by-\bx_c)|^2)\rho(\bx)\rho(\by)\rd{\bx}\rd{\by} \\
= &  2m_0 \int (|\bu(\bx)-\bu_c|^2 + a|\bx-\bx_c|^2)\rho(\bx)\rd{\bx}
\end{split}
\end{equation}
Thus, at time $t_0$,
\[
\int_{|\bu(\bx)-\bu_c|^2 + a|\bx-\bx_c|^2 \ge \frac{\epsilon_1}{4}} \rho(\bx)\rd{\bx} \le \frac{4}{\epsilon_1}\int (|\bu(\bx)-\bu_c|^2 + a|\bx-\bx_c|^2)\rho(\bx)\rd{\bx} \le \frac{4}{\epsilon_1} \frac{1}{2m_0} \frac{m_0\eta \epsilon_1^2}{2} = \eta\epsilon_1
\]
Therefore, we can take $S_2 := \{\bx:|\bu(\bx)-\bu_c|^2 + a|\bx-\bx_c|^2 \ge \epsilon_1/4\}$, and (\ref{eps1}) is satisfied. Then for any $\bx,\by \in S_1:= \textnormal{supp }\rho \backslash S_2$, one has
\[
\begin{split}
|\bu(\bx)-\bu(\by)|^2+a|\bx-\by|^2 \le & |(\bu(\bx)-\bu_c)-(\bu(\by)-\bu_c)|^2+a|(\bx-\bx_c)-(\by-\bx_c)|^2 \\
\le & 2(|\bu(\bx)-\bu_c|^2 +a |\bx-\bx_c|^2 + |\bu(\by)-\bu_c|^2 +a |\by-\bx_c|^2) \\
\le & 4\frac{\epsilon_1}{4} = \epsilon_1
\end{split}
\]
which means (\ref{eps2}) is also satisfied.

\end{proof}

\begin{proof}[Proof of Theorem \ref{thm_flocking3}]

We start by a hypocoercivity argument on the energy estimate. Using the notation in the proof of Theorem \ref{thm_flocking2},
\begin{equation}
\begin{split}
& \frac{\rd}{\rd{t}} \langle \bx-\by,\bu(\bx)-\bu(\by) \rangle \\
= & \int\int \Big[(\bx-\by)\cdot(\bu(\bx)-\bu(\by))(-\nabla_\bx\cdot(\rho(\bx)\bu(\bx))\rho(\by)-\nabla_\by\cdot(\rho(\by)\bu(\by))\rho(\bx)) \\
& + (\bx-\by)\cdot\Big( -\bu(\bx)\cdot\nabla_\bx \bu(\bx) + \int \phi(\bx-\bz)(\bu(\bz)-\bu(\bx))\rho(\bz)\rd{\bz} - \nabla\FP(\bx) \\
& -\bu(\by)\cdot\nabla_\by \bu(\by) + \int \phi(\by-\bz)(\bu(\bz)-\bu(\by))\rho(\bz)\rd{\bz} - \nabla\FP(\by) \Big)  \rho(\bx)\rho(\by)\Big]\rd{\bx}\rd{\by} \\
= & |\bu(\bx)-\bu(\by)|^2 + \int\int (\bx-\by)\cdot\Big(\int \phi(\bx-\bz)(\bu(\bz)-\bu(\bx))\rho(\bz)\rd{\bz} \\
& + \int \phi(\by-\bz)(\bu(\bz)-\bu(\by))\rho(\bz)\rd{\bz} \Big)  \rho(\bx)\rho(\by)\rd{\bx}\rd{\by}  - \langle \bx-\by,\nabla\FP(\bx)-\nabla\FP(\by) \rangle\\ 
\le & |\bu(\bx)-\bu(\by)|^2 - a|\bx-\by|^2 + 2(\frac{a}{4}|\bx-\by|^2 + \frac{m_0^2\phi_+^2}{a}|\bu(\bx)-\bu(\by)|^2) \\
= &  - \frac{a}{2}|\bx-\by|^2 +  \Big(1+\frac{2m_0^2\phi_+^2}{a}\Big)|\bu(\bx)-\bu(\by)|^2 \\
\end{split}
\end{equation}
where we used
\begin{equation}
\begin{split}
& \left|  \int\int (\bx-\by)\cdot\int \phi(\bx-\bz)(\bu(\bz)-\bu(\bx))\rho(\bz)\rd{\bz}\rho(\bx)\rho(\by)\rd{\bx}\rd{\by} \right| \\
\le & \phi_+c_1|\bx-\by|^2 + \frac{\phi_+}{4c_1}\int\int \left(\int |(\bu(\bz)-\bu(\bx))|\rho(\bz)\rd{\bz}\right)^2\rho(\bx)\rho(\by)\rd{\bx}\rd{\by} \\
\le & \phi_+c_1|\bx-\by|^2 + \frac{\phi_+}{4c_1}\int\int m_0\int |(\bu(\bz)-\bu(\bx))|^2\rho(\bz)\rd{\bz} \rho(\bx)\rho(\by)\rd{\bx}\rd{\by} \\
\le & \phi_+c_1|\bx-\by|^2 + \frac{m_0^2\phi_+}{4c_1}|\bu(\bx)-\bu(\by)|^2
\end{split}
\end{equation}
with $c_1=\nicefrac{a}{4\phi_+}$.
Combined with the energy estimate (\ref{eq:decay}), we obtain, for any $c>0$,
\begin{align*}
\frac{\rd}{\rd{t}} \big(\Et(t) & + c\langle \bx-\by,\bu(\bx)-\bu(\by) \rangle \big)  \le -\Big(\frac{\phi_-}{2} - c\big(1+\frac{2m_0^2\phi_+^2}{a}\big)\Big) |\bu(\bx)-\bu(\by)|^2 - \frac{ca}{2}|\bx-\by|^2.
\end{align*}
Then, setting 
\begin{equation}
c := \min\Big\{ \frac{\nicefrac{\displaystyle \phi_-}{\displaystyle 2}}{\displaystyle 1+\nicefrac{\displaystyle 2m_0^2\phi_+^2}{\displaystyle a}+\nicefrac{\displaystyle 1}{\displaystyle 2}},\frac{\sqrt{a}}{8m_0}\Big\}
\end{equation}
we have
\[
\frac{\rd}{\rd{t}} (\Et(t) + c\langle \bx-\by,\bu(\bx)-\bu(\by) \rangle ) \le -\frac{c}{2}( |\bu(\bx)-\bu(\by)|^2 +a|\bx-\by|^2) = -\frac{c}{2}\cEt(t).
\]
Notice that since $\FP(\bx)\ge \frac{a}{2}|\bx|^2$,
\[
\begin{split}
\langle \bx-\by,\bu(\bx)-\bu(\by) \rangle \le & \frac{1}{2\sqrt{a}}(a|\bx-\by|^2 + |\bu(\bx)-\bu(\by)|^2) \\
\le & \frac{2m_0}{\sqrt{a}}\int (a|\bx|^2 + |\bu(\bx)|^2)\rho(\bx)\rd{\bx} \le \frac{4m_0}{\sqrt{a}}\Et(t)
\end{split}
\]
Therefore $\Et(t) + c\langle \bx-\by,\bu(\bx)-\bu(\by) \rangle \ge 0$,which in turn implies that
$\displaystyle \int_0^\infty \cEt(t)
\rd{t} =: C_0 < \infty$.

Next, for any fixed $t_1>0$, there exists $t_0\le t_1$ such that
$\displaystyle \cEt(t_0) \le \frac{C_0}{t_1}$; 
(otherwise the integral $\int_0^{t_1} \cEt(t)\rd{t}$ would exceed $C_0$). Lemma \ref{lem_par} implies that there exists a partition at $t=t_0$ satisfying (\ref{eps1}) and (\ref{eps2}), with $\epsilon_1$ given by
$\displaystyle \epsilon_1 = \sqrt{\frac{2C_0}{m_0\eta t_1}}$.
If $t_1$ is large enough, then $\epsilon_1$ is small enough, so that we can apply Lemma \ref{lem_sep} to get that (\ref{eps3}) holds for all $t\ge t_0$. In particular, (\ref{eps3}) holds for $t=t_1$. Therefore, by using (\ref{eps3}) for pairs $(\bx,\by)$ with $\bx,\by\in S_1(t_1)$ and the uniform bound (\ref{eq:ubound}) for other pairs, we obtain ($u_{max}$ denoting the RHS of (\ref{eq:ubound}))
\begin{equation}
\cEt(t_1) \le m_0^2 \epsilon_1 + 2m_0\eta \epsilon_1 \cdot 4(1+a)u_{max}^2 = C\epsilon_1
\end{equation}
and the proof is finished by noticing that $\epsilon_1 = {\mathcal O}(1/\sqrt{t_1})$ for large $t_1$.
\end{proof}

\section{Proof of main results --- existence of global smooth solutions}

\subsection{The one-dimensional case}
The proof of the existence of global smooth solutions for 1d  follows the technique of~\cite{CCTT2016}: we analyze the ODE satisfied by the quantity $\partial_x u + \phi*\rho$ along characteristics.
\begin{proof}[Proof of Theorem \ref{thm_1dsmooth}]
Write $\texttt{d}:=\partial_x u$. Differentiate the second equation of (\ref{eq1}) with respect to $x$ to get
\begin{equation}\label{eq1}
\begin{split}
\partial_t \rho + u\partial_x \rho & = -\rho \texttt{d} \cr
 \partial_t \texttt{d} + u\partial_x \texttt{d} + \texttt{d}^2 & = -u\int \partial_x\phi(x-y)\rho(y)\rd{y} - \int \phi(x-y)\partial_t\rho(y)\rd{y}\cr
 & \quad - \texttt{d}\int \phi(x-y)\rho(y)\rd{y} - \FP''(x)
\end{split}
\end{equation}
Expressed in terms of $\texttt{e}:=\texttt{d}+\phi*\rho$ and 
 the time derivative along characteristics denoted by ${}'$, then \eqref{eq1} reads
\begin{equation}\label{de1d}
\begin{split}
& \rho' = -\rho(\texttt{e}-\phi*\rho) \\
& \texttt{e}' = -\texttt{e}(\texttt{e}-\phi*\rho) - \FP''.
\end{split}
\end{equation}
If $\texttt{e}>0$, then by (\ref{1d_assu1}),
\[
\texttt{e}' \ge -\texttt{e}(\texttt{e}-m_0\phi_-) - A = -\left(\texttt{e}-\frac{m_0\phi_-}{2}\right)^2 + \left(\frac{(m_0\phi_-)^2}{4} - A\right).
\]
Then by (\ref{1d_assu2}), one has
\[
\texttt{e}'   > 0,  \quad \text{ for } \frac{m_0\phi_-}{2} - \sqrt{\frac{(m_0\phi_-)^2}{4} - A} < \texttt{e} < \frac{m_0\phi_-}{2} + \sqrt{\frac{(m_0\phi_-)^2}{4} - A}. 
\]
By (\ref{1d_assu3}), initially $\displaystyle \texttt{e}>\frac{m_0\phi_-}{2} - \sqrt{\frac{(m_0\phi_-)^2}{4} - A}$ for all $x$. Therefore the same inequality persists for all time. 

Also notice that if $\texttt{e} \ge 2m_0\phi_+$ then 
$\texttt{e}' \le -\nicefrac{\texttt{e}^2}{2} - a$, which implies $\texttt{e}$ is bounded above by $\texttt{e} \le \max\{\max_x \texttt{e}_0,2m_0\phi_+,\sqrt{\max\{0,-2a\}}\}$.  Since $\phi*\rho$ is bounded above and below, this implies that $\partial_x u$ is uniformly bounded, and thus global smooth solution exists.
\end{proof}

\begin{proof}[Proof of Theorem \ref{thm_1dblowup}]
We start from (\ref{de1d}), the dynamic of $\texttt{e}$, which is derived in the previous proof. We analyze the sign of $\texttt{e}'$ in the cases of positive and negative $\texttt{e}$:
\begin{itemize}
\item If $\texttt{e}\ge 0$, then 
\begin{equation}
\texttt{e}' \le -\texttt{e}(\texttt{e}-m_0\phi_+) - a = -\left(\texttt{e}-\frac{m_0\phi_+}{2}\right)^2 + \left(\frac{(m_0\phi_+)^2}{4} - a\right)
\end{equation}

\begin{itemize}
\item If (\ref{assuB_1}) holds, then $\texttt{e}'<0$.

\item If (\ref{assuB_1}) does not hold, then if 
\begin{equation}\label{assuB_21}
\texttt{e} < \frac{m_0\phi_+}{2} - \sqrt{\frac{(m_0\phi_+)^2}{4} - a}
\end{equation}
then $\texttt{e}'<0$.
\end{itemize}

\item If $\texttt{e}<0$ then
\begin{equation}
\texttt{e}' \le -\texttt{e}(\texttt{e}-m_0\phi_-) - a = -\left(\texttt{e}-\frac{m_0\phi_-}{2}\right)^2 + \left(\frac{(m_0\phi_-)^2}{4} - a\right)
\end{equation}

\begin{itemize}
\item If $a>0$, then $\texttt{e}'<0$.

\item If $a\le 0$, then if 
\begin{equation}\label{assuB_31}
\texttt{e} < \frac{m_0\phi_-}{2} - \sqrt{\frac{(m_0\phi_-)^2}{4} - a}
\end{equation}
then $\texttt{e}'<0$. 
\end{itemize}
\end{itemize}

Notice that for all the $\texttt{e}'<0$ cases above, we actually have $\texttt{e}'<-\epsilon<0$. Therefore, as long as one stays in the $\texttt{e}'<0$ cases, $\texttt{e}$ will keep decreasing until it is negative enough so that the $-\texttt{e}^2$ term blows it up. Therefore, we have the following situations where we can guarantee a finite time blow-up:
\begin{itemize}
\item If (\ref{assuB_1}) holds, then any negative values of $\texttt{e}$ will have $\texttt{e}'<0$ since $a>0$, and any positive values of $\texttt{e}$ will have $\texttt{e}'<0$.
\item If (\ref{assuB_1}) does not hold but $a> 0$ and (\ref{assuB_2}) holds (which means (\ref{assuB_21}) holds initially), then (\ref{assuB_21}) will propagate since $\texttt{e}'<0$ for positive or negative values of $\texttt{e}$.
\item If (\ref{assuB_1}) does not hold and $a\le 0$ but (\ref{assuB_3}) holds (which means (\ref{assuB_31}) holds initially: in particular, $\texttt{e}$ starts with negative values), then (\ref{assuB_31}) will propagate since $\texttt{e}'<0$ (because $\texttt{e}$ stays negative).  
\end{itemize}

\end{proof}

\subsection{The two-dimensional case}
We  follow~\cite{HeT2017}, tracing  the dynamics of the matrix $M_{ij}=\partial_ju_i$ associated with the solution to \eqref{eq}.  Since most steps  are the same as in \cite[Theorem 2.1]{HeT2017} except for the additional external potential term on the right of \eqref{eq}, we outline the derivation along the same steps as in \cite{HeT2017} while omitting excessive details. 

\smallskip\noindent
STEP 1: $M$ satisfies
\begin{equation}\label{M}
\partial_t M + \bu \cdot\nabla M + M^2 = -(\phi*\rho)M + R - \nabla^2 \FP
\end{equation}
where
\begin{equation}\label{eq:Rij}
R_{ij} = \partial_j \phi * (\rho u_i) - u_i(\partial_j \phi * \rho)
\end{equation}
 The divergence $\texttt{d}=\nabla\cdot\bu$,  satisfies
\begin{equation}
\partial_t \texttt{d} + \bu\cdot\nabla\texttt{d} + \text{Tr} M^2 = -(\phi*\rho)\texttt{d} + \text{Tr} R - \Delta \FP.
\end{equation}
The two traces in this equation are  evaluated as follows. By \eqref{eq:Rij}, $\text{Tr} R = -(\phi*\rho)'$; also, $\text{Tr} M^2 \equiv \frac{1}{2}\big(\texttt{d}^2 + \eta_M^2\big)$ where $\eta_M$ is the spectral gap of the two eigenvalues of $M$. We find
\begin{equation}
(\texttt{d}+\phi*\rho)' = -\frac{1}{2}\eta_M^2 -\frac{1}{2}\texttt{d}(\texttt{d}+2\phi*\rho) - \Delta \FP
\end{equation}
Decompose $M$ into its symmetric and anti-symmetric parts, $M=S+\Omega$, then $\eta_M^2 = \eta_S^2 - 4\omega^2$
where $\eta_S$ is the spectral gap  of $S$ and $\omega=(\partial_1u_2-\partial_2u_1)/2$ is the scaled vorticity.
Then by introducing 
$\texttt{e} = \texttt{d}+\phi*\rho$
we finally end up with 
\begin{equation}\label{de}
\texttt{e}' = \frac{1}{2}( 4\omega^2+ (\phi*\rho)^2 - \eta_S^2 - \texttt{e}^2  - 2\Delta \FP), \qquad \texttt{e} := \texttt{d}+\phi*\rho
\end{equation}

\smallskip\noindent
STEP 2: The `$\texttt{e}$-equation' is complemented by the dynamics of the spectral gap $\eta_S$. To this end, we follow the spectral dynamics of  $S$,
\[
S' + S^2 = \omega^2 I - (\phi*\rho)S + R_{sym} - \nabla^2 \FP,\quad R_{sym} = \frac{1}{2}(R+R^\top);
\]
where $I$ stands for the identity matrix.
The dynamics of the eigenvalues $\mu_i$ of $S$ is given by
\[
\mu_i' + \mu_i^2 = \omega^2 - (\phi*\rho)\mu_i + \langle \bs_i, R_{sym}\bs_i \rangle - \langle \bs_i, \nabla^2\FP \bs_i \rangle
\]
where $\bs_1,\bs_2$ are the orthonormal eigenpair of $S$. Taking their difference,
\begin{equation}\label{etaS}
\eta_S' + \texttt{e}\eta_S = q := \langle \bs_2, R_{sym}\bs_2 \rangle - \langle \bs_1, R_{sym}\bs_1 \rangle - \langle \bs_2, \nabla^2 \FP\bs_2 \rangle + \langle \bs_1, \nabla^2 \FP\bs_1 \rangle.
\end{equation}

\smallskip\noindent
STEP 3: We need to estimate $\eta_S$ based on (\ref{etaS}). A good estimate of $\eta_S$ will give a non-negative lower bound of $\texttt{e}$. We will conduct this estimate for the quadratic potential and general convex potentials in different ways in the following subsections. 

\smallskip\noindent
STEP 4: Finally we need an upper bound of $\texttt{e}$. The dynamics of $\omega$ is independent of the symmetric forcing term $\nabla^2\FP$,
\begin{equation}\label{omega}
\omega' + \texttt{e}\omega = \frac{1}{2}\text{Tr}(JR),\quad J = \left[\begin{array}{cc}0 & -1\\ 1 & 0\end{array}\right],
\end{equation}
 Therefore we can bound $\omega$ in the same way as we bound $\eta_S$, and this yields an upper bound of $\texttt{e}$. This would conclude the proof of the uniform boundedness of $\texttt{d}=\nabla\cdot\bu$. Combined with the uniform boundedness of $\eta_S$ and $\omega$, we get the uniform boundedness of $\nabla\bu$.

\medskip
\paragraph{{\bf $\bullet$ Quadartic potentials}}. We  elaborate STEP 3 and STEP 4 for the quadratic potential. For the 2D case with quadratic potential, $\nabla^2\FP$ is constant multiple of the identity matrix, and thus the last two terms in (\ref{etaS}) cancel. Also, we already know from proposition \ref{prop_infty1} that the solution flocks at exponential rate, in the sense of $L^\infty$. This enables us to estimate $\eta_S$ in the same way as in~\cite{HeT2017}.

\begin{proof}[Proof of Theorem \ref{thm_2dsmooth1}]
For $\FP(\bx)=\frac{a}{2}|\bx|^2$, the $q$ defined in (\ref{etaS}) becomes
\begin{equation}
q = \langle \bs_2, R_{sym}\bs_2 \rangle - \langle \bs_1, R_{sym}\bs_1 \rangle
\end{equation}
with $R$ satisfying the estimate
\[
|R| \le 8m_0|\phi'|_\infty\sqrt{\Ckitty \cdot\cEiin} \cdot e^{-\lambda t/4},\quad \forall \bx
\]
Therefore, since $\bs_1,\bs_2$ are unit vectors,
\[
|q| \le 16m_0|\phi'|_\infty\sqrt{\Ckitty \cdot\cEiin} \cdot e^{-\lambda t/4}.\quad \forall \bx
\]
Hence, as long as $\texttt{e}$ remains non-negative, $\eta_S$ is bounded by constant:
\begin{equation}
|\eta_S| \le \max_\bx |(\eta_S)_0(\bx) | + \frac{64}{\mixed}m_0|\phi'|_\infty\sqrt{\Ckitty \cdot\cEiin}  = \max_\bx |(\eta_S)_0(\bx) | + \Cmeow \cdot\sqrt{\cEiin}.
\end{equation}

\smallskip\noindent
STEP 3: The $\texttt{e}$ equation \eqref{de} implies
\begin{equation}\label{de1_1}
\texttt{e}' \ge \frac{1}{2}(c_1^2 - \texttt{e}^2)
\end{equation}
with $c_1$ defined by (\ref{c1}).
In this case, (\ref{de1_1}) implies that $\texttt{e}$ remains non-negative if $\texttt{e}_0(\bx) \ge 0$ for all $\bx$, as assumed in (\ref{c1_cond1}).

\smallskip\noindent
STEP 4: Similarly we obtain from (\ref{omega}) that $\omega$ is uniformly bounded:
\begin{equation}
|\omega| \le \max_\bx \omega_0(\bx)+ \frac{32}{\mixed}m_0|\phi'|_\infty  \sqrt{\Ckitty \cdot\cEiin}=: \omega_{max}
\end{equation}
Then, since $\Delta \FP = 2a > 0$, (\ref{de}) shows
$\texttt{e}' \le \frac{1}{2}(4\omega_{max}^2 + m_0^2\phi_+^2 - \texttt{e}^2)$,
and we end up with  the uniform upper bound,
$\texttt{e} \le \max\left\{\max_\bx \texttt{e}_0(\bx), \sqrt{4\omega_{max}^2 + m_0^2\phi_+^2}\right\}$.
\end{proof}

\medskip
\paragraph{{\bf $\bullet$ General convex potentials}}
Recall  that in the case of quadratic potential, the last two terms in (\ref{etaS}) cancel since $\nabla^2\FP$ is a constant multiple of the identity matrix. Also, $R_{sym}$ has an exponential decay estimate by the $L^\infty$ flocking result.  These two facts enabled us to estimate $\eta_S$ by $|(\eta_S)_0|+\int_0^\infty |q(t)| \rd{t}$, without making use of the good term $\texttt{e}\eta_S$.

However, for general convex potentials, we lack  a flocking estimate, and the last two terms in (\ref{etaS}) do not cancel. Therefore, $q$, the RHS of (\ref{etaS}), do not have a time decay estimate. The best we can hope is to bound $q$ uniformly in time by a constant $\Cmi$, and then propagate a {\it positive} lower bound $c_2$ of $\texttt{e}$, in order to control $\eta_S$ by $\max\{|(\eta_S)_0| ,\, \frac{\Cmi }{c_2}\}$. 

\begin{proof}[Proof of Theorem \ref{thm_2dsmooth2}]
We start from (\ref{etaS}). Since $\bs_i$ are normalized, $q$ is controlled by Proposition \ref{prop_infty2} and assumption (\ref{assu_A}) as
\begin{equation}
|q| \le 8m_0|\phi'|_{\infty}u_{max} + 2A =\Cmi 
\end{equation}
where $\Cmi $ is as defined in (\ref{Cmi}). Hence, assume we have the lower bound (which is true initially, by assumption (\ref{c1_cond2}))
\begin{equation}\label{c2}
\texttt{e} \ge \sqrt{C_A - \sqrt{C_A^2-\Cmi ^2}} =: c_2 > 0
\end{equation}
(the quantity inside the inner square root is positive, by assumption (\ref{Cmi})) where $C_A$ as defined in (\ref{Cmi}), then $\eta_S$ is bounded by constant:
\begin{equation}
|\eta_S| \le \max\Big\{\max_\bx |(\eta_S)_0(\bx) | ,\, \frac{\Cmi }{c_2}\Big\} := \eta_{S,max}
\end{equation}

\smallskip\noindent
STEP 3: (\ref{de}) implies
\begin{equation}\label{de1}
\texttt{e}' \ge \frac{1}{2}(c_1^2 - \texttt{e}^2), \qquad c_1:=\sqrt{m_0^2\phi_-^2 - \eta_{S,max}^2 - 4A}=\sqrt{2C_A - \eta_{S,max}^2},
\end{equation}
provided  the quantity inside the square root on the right is positive. In fact, assumption (\ref{etaS_cond2}) gives
\[
2C_A - \max_\bx |(\eta_S)_0(\bx) |^2  \ge C_A  - \sqrt{C_A^2-\Cmi ^2} = c_2^2
\]
and by (\ref{c2}), $2C_A - (\frac{\Cmi }{c_2})^2=  c_2^2$.
Thus we have $2C_A - \eta_{S,max}^2  = c_2^2$, and therefore $c_1$ is well-defined and coincides with $c_1= c_2 > 0$.
With this,  \eqref{de1} now reads
$\texttt{e}' \ge \nicefrac{1}{2}(c_2^2 - \texttt{e}^2)$
and hence $\texttt{e}$ is increasing whenever $\texttt{e} \le c_2$. 
This means the initial bound $\texttt{e} \ge c_2$ can be propagated for all time.

\smallskip\noindent
STEP 4: Similarly we obtain from (\ref{omega}) that $\omega$ is uniformly bounded:
\[
|\omega| \le \max\left\{ \max_\bx |\omega_0(\bx)|,\frac{4m_0|\phi'|_{\infty} u_{max}}{c_2}  \right\} =: \omega_{max}
\]
Then (\ref{de}) shows, since $|\Delta \FP| \le 2A$,
$\texttt{e}' \le \frac{1}{2}(4\omega_{max}^2 + m_0^2\phi_+^2 + 4A - \texttt{e}^2)$.
Thus we get the upper bound, $\texttt{e} \le \max\left\{\max_\bx \texttt{e}_0(\bx), \sqrt{4\omega_{max}^2 + m_0^2\phi_+^2+4A}\right\}$.
\end{proof}

\end{document}